\documentclass[11pt]{article}

\usepackage{amsmath,amssymb,amsthm,dsfont,mathrsfs,stmaryrd}
\usepackage[T1]{fontenc}
\usepackage[utf8]{inputenc}
\usepackage{titlesec,titletoc}
\usepackage[all,cmtip]{xy}
\usepackage[fleqn,tbtags]{mathtools}
\usepackage{mathabx}
\usepackage{url}
\usepackage[numbers]{natbib}
\usepackage{tikz,tikz-cd}
\usepackage{enumerate}
\usepackage{wasysym} 
\usepackage{easybmat} 
 \usetikzlibrary{decorations.pathmorphing} 
 \usetikzlibrary{decorations.markings}

\theoremstyle{plain}
\newtheorem{theo}{Theorem}[subsection]
\newtheorem{prop}[theo]{Proposition}
\newtheorem{cor}[theo]{Corollary}

\newtheorem{conj}[theo]{Conjecture}
\newtheorem{lem}[theo]{Lemma}

\theoremstyle{definition}
\newtheorem{defi}[theo]{Definition}
\newtheorem{ex}[theo]{Example}

\theoremstyle{remark}
\newtheorem{rem}[theo]{Remark}

\DeclareMathOperator{\supp}{supp}

\DeclareMathOperator{\id}{Id}
\DeclarePairedDelimiter{\abs}{\lvert}{\rvert}
\newcommand{\Uqg}{\mathcal{U}_{q}(L\mathfrak{g})}
\newcommand{\Uqb}{\mathcal{U}_{q}(\mathfrak{b})}

\newcommand{\cO}{\mathcal{O}}
\newcommand{\Op}{\mathcal{O}^{+}}
\newcommand{\Om}{\mathcal{O}^{-}}

\newcommand{\Plr}{P_{\ell}^{\mathfrak{r}}}
\newcommand{\Y}{\mathscr{Y}}
\newcommand{\T}{\mathscr{T}}

\newcommand{\Ctt}{\mathbb{C}(t^{1/2})}
\newcommand{\C}{\mathscr{C}}
\newcommand{\qt}{(q,t)}
\newcommand{\g}{\mathfrak{g}}
\newcommand{\Z}{\mathbb{Z}}
\newcommand{\N}{\mathbb{N}}
\newcommand{\Q}{\mathbb{Q}}
\newcommand{\A}{\mathcal{A}}
\newcommand{\Yt}{\mathcal{Y}_t}

\newcommand{\E}{\mathcal{E}}
\newcommand{\F}{\mathcal{F}}
\newcommand{\Kp}{K_0(\cO^+_\Z)}
\newcommand{\Ktp}{K_t(\cO^+_\Z)}
\newcommand{\e}{\textbf{e}}
\newcommand{\f}{\textbf{f}}

\newcommand{\J}{\mathcal{J}}

\numberwithin{equation}{section}
\author{Léa Bittmann}
\title{Quantum Grothendieck rings as quantum cluster algebras.}

\begin{document}

\tikzset{->-/.style={decoration={
  markings,
  mark=at position .5 with {\arrow{>}}},postaction={decorate}}}

\maketitle

\begin{abstract}
We define and construct a quantum Grothendieck ring for a certain monoidal subcategory of the category $\cO$ of representations of the quantum loop algebra introduced by Hernandez-Jimbo. We use the cluster algebra structure of the Grothendieck ring of this category to define the quantum Grothendieck ring as a quantum cluster algebra. When the underlying simple Lie algebra is of type $A$, we prove that this quantum Grothendieck ring contains the quantum Grothendieck ring of the category of finite-dimensional representations of the associated quantum affine algebra. In type $A_1$, we identify remarkable relations in this quantum Grothendieck ring.
\end{abstract}

\tableofcontents

%\newpage

	\section{Introduction}
	
Let $\g$ be a simple Lie algebra of Dynkin type $A$, $D$ or $E$ (also called simply laced types), and let $L\g=\g\otimes\mathbb{C}[t^{\pm 1}]$ be the loop algebra of $\g$. For $q$ a generic complex number, Drinfeld \citep{ANRY} introduced a $q$-deformation of the universal enveloping algebra $U(L\g)$ of $L\g$ called the \emph{quantum loop algebra} $\Uqg$. It is a $\mathbb{C}$-algebra with a Hopf algebra structure, and the category $\C$ of its finite-dimensional representations is a monoidal category. The category $\C$ was studied extensively, in particular to build solutions to the quantum Yang-Baxter equation with spectral parameter (see \citep{YBe} for a detailed review). 

Using the so-called "Drinfeld-Jimbo" presentation of the quantum loop algebra, one can define a \emph{quantum Borel subalgebra} $\Uqb$, which is a Hopf subalgebra of $\Uqg$. We are here interested in studying a category $\cO$ of representations of $\Uqb$ introduced by Hernandez-Jimbo \citep{ARDRF}. The category $\cO$ contains all finite-dimensional $\Uqb$-modules, as well as some infinite-dimensional representations, however, with finite-dimensional weight spaces. In particular, this category $\cO$ contains the \emph{prefundamental representations}. These are a family of infinite dimensional simple $\Uqb$-modules, which first appeared in the work of Bazhanov, Lukyanov, Zamolodchikov \cite{ISCFT} for $\g=\mathfrak{sl}_2$ under the name $q$\emph{-oscillator representations}. 

These prefundamental representations were also used by Frenkel-Hernandez \citep{BRSQIM} to prove Frenkel-Reshetikhin’s conjecture on the spectra of
quantum integrable systems \citep{QCRQAA}. More precisely, quantum integrable systems are studied via a partition function $Z$, which in turns can be scaled down to the study of the eigenvalues $\lambda_j$ of the transfer matrix $T$. For the 6-vertex (and 8-vertex) models, \citep{BAXT} showed that the eigenvalues of $T$ have the following remarkable form:
\begin{equation}\label{TQ}
 \lambda_j = A(z) \frac{Q_j(zq^{-2})}{Q_j(z)} + D(z) \frac{Q_j(zq^{2})}{Q_j(z)},
\end{equation}
where $q$ and $z$ are parameters of the model, the functions $A(z)$, $D(z)$ are universal, and $Q_j$ is a polynomial. This relation is called the \emph{Baxter relation}. In the context of representation theory, relation (\ref{TQ}) can be categorified as a relation in the Grothendieck ring of the category $\cO$. For $\g=\mathfrak{sl}_2$, if $V$ is the two-dimensional simple representation of $\Uqg$ of highest loop-weight $Y_{aq^{-1}}$, then
\begin{equation}\label{TQ'}
[V\otimes L^+_{a}] = [\omega_1][L^+_{aq^{-2}}] +[-\omega_1][L^+_{aq^2}],
\end{equation}
where $[\pm\omega_1]$ are one-dimensional representations of weight $\pm\omega_1$ and $L^+_{a}$ denotes the positive prefundamental representation of quantum parameter $a$.

Frenkel-Reshetikhin’s conjecture stated that for more general quantum integrable systems, constructed via finite-dimensional representations of the quantum affine algebra $\mathcal{U}_q(\hat{\g})$ (of which the quantum loop algebra is a quotient) the spectra had a similar form as relation (\ref{TQ}). 

Let $t$ be an indeterminate. The Grothendieck ring of the category $\C$ has an interesting $t$-deformation called the \emph{quantum Grothendieck ring}, which belongs to a non-commutative quantum torus $\Y_t$. The quantum Grothendieck ring was first studied by Nakajima \citep{QVtA} and Varagnolo-Vasserot \citep{PSqGR} in relation with quiver varieties. Inside this ring, one can define for all simple modules $L$ classes $[L]_t$, called $\qt$\emph{-characters}. Using these classes, Nakajima was able to compute the characters of the simple modules $L$, which were completely not accessible in general, thanks to a Kazhdan-Lusztig type algorithm. 

One would want to extend these results to the context of the category $\cO$, with the ultimate goal of (algorithmically) computing characters of all simple modules in $\cO$. In order to do that, one first needs to build a quantum Grothendieck ring $K_t(\cO)$ inside which the classes $[L]_t$ can be defined. 

Another interesting approach to this category $\cO$ is its \emph{cluster algebra} structure (see below). Hernandez-Leclerc \citep{MCCA} first noticed that the Grothendieck ring of a certain monoidal subcategory $\C_1$ of the category $\C$ of finite-dimensional $\Uqg$-modules had the structure of a cluster algebra. Then, they proved \citep{CABS} that the Grothendieck ring of a certain monoidal subcategory $\cO^+_\Z$  of the category $\cO$ had a cluster algebra structure, of infinite rank, for which one can take as initial seed the classes of the positive prefundamental representations (the category $\cO^+_\Z$ contains the finite-dimensional representations and the positive prefundamental representations whose spectral parameter satisfy an integrality condition).  Moreover, some exchange relations coming from cluster mutations appear naturally. For example, the Baxter relation (\ref{TQ'}) is an exchange relation in this cluster algebra.

In order to construct of quantum Grothendieck ring for the category $\cO$, the approaches used previously are not applicable anymore. The geometrical approach of Nakajima and Varagnolo-Vasserot (in which the $t$-graduation naturally comes from the graduation of cohomological complexes) requires a geometric interpretation of the objects in the category $\cO$, which has not yet been found. The more algebraic approach consisting of realizing the (quantum) Grothendieck ring as an invariant under a sort of Weyl symmetry, which allowed Hernandez to define a quantum Grothendieck ring of finite-dimensional representations in non-simply laced types, is again not relevant for the category $\cO$. Only the cluster algebra approach yields results in this context.

In this paper, we propose to build the quantum Grothendieck of the category $\cO^+_\Z$ as a \emph{quantum cluster algebra}. Quantum cluster algebras are non-commutative versions of cluster algebras, they live inside a quantum torus, generated by the initial variables, together with $t$-commuting relations:
\begin{equation}
X_i \ast X_j = t^{\Lambda_{ij}} X_j\ast X_i.
\end{equation}
First of all, one has to build such a quantum torus, and check that it contains the quantum torus $\Y_t$ of the quantum Grothendieck ring of the category $\C$. This is proven as the first result of this paper (Proposition \ref{propinclusiontores}). 

Next, one has to show that this quantum torus is compatible with a quantum cluster algebra structure based on the same quiver as the cluster algebra structure of the Grothendieck ring $K_0(\cO^+_\Z)$. In order to do that, we exhibit (Proposition \ref{propComp}) a \emph{compatible pair}. From then, the quantum Grothendieck ring $\Ktp$ is defined as the quantum cluster algebra defined from this compatible pair.

We then conjecture (Conjecture \ref{conj}) that this quantum Grothendieck ring $\Ktp$ contains the quantum Grothendieck ring $K_t(\C_\Z)$. We propose to demonstrate this conjecture by proving that $\Ktp$ contains the $\qt$-characters of the fundamental representations $[L(Y_{i,q^r})]_t$, as they generate $K_t(\C_\Z)$. We state in Conjecture \ref{conjchir} that these objects can be obtained in $\Ktp$ as quantum cluster variables, by following the same finite sequences of mutations used in the classical cluster algebra $K_0(\cO^+_\Z)$ to obtain the $[L(Y_{i,q^r})]$. Naturally, Conjecture \ref{conjchir} implies Conjecture \ref{conj}.
Finally, we prove Conjecture \ref{conjchir} (and thus Conjecture \ref{conj}) in the case where the underlying simple Lie algebra $\g$ is of type $A$ (Theorem \ref{theoA}). The proof is based on the thinness property of the fundamental representations in this case. When $\g=\mathfrak{sl}_2$, some explicit computations are possible. For example, we give a quantum version of the Baxter relation (\ref{baxterq}), for all $r\in \Z$,
\begin{equation*}
[V_{q^{2r-1}}]_t\ast[L_{1,q^{2r}}^+]_t = t^{-1/2}[\omega_1][L_{1,q^{2r-2}}^+]_t + t^{1/2}[-\omega_1][L_{1,q^{2r+2}}^+]_t. 
\end{equation*}
Additionally, we realize a part of the quantum cluster algebra we built as a quotient of the quantum group $\mathcal{U}_q(\mathfrak{sl}_2)$. This is a reminiscence of the result of Qin \citep{QIN} who constructed $\mathcal{U}_q(\g)$ as a quotient of the Grothendieck ring arising from certain cyclic quiver varieties.

The paper is organized as follows. The first three sections are mostly reminders. In Section \ref{sectclusterqcluster} we recall some background on cluster algebras and quantum cluster algebras, including some recent and important results, such as the positivity Theorem in Section \ref{sectpos}, which we require later on. In Section \ref{sectCartan} we introduce some notations, the usual notations for the Cartan data associated to a simple Lie algebra, as well as what we call \emph{quantum Cartan data}, which is related to the quantum Cartan matrix and its inverse. In Section \ref{sectO} we review some results for the category $\cO$, its subcategories $\cO^\pm$ and $\cO^\pm_\Z$ and their Grothendieck rings. In Section \ref{sectQuantumTori}, after recalling the definition of the quantum torus $\Y_t$, we define the quantum torus $\T_t$ in which $\Ktp$ will be constructed and we prove the inclusion of the quantum tori. In Section \ref{sectQuantumGrothendieckRings} we prove that we have all the elements to build a quantum cluster algebra and we define the quantum Grothendieck ring $\Ktp$. In the concluding Section \ref{sectPropertiesKtp}, we state some properties of the quantum Grothendieck ring. We present the two conjectures regarding the inclusion of the quantum Grothendieck rings in Section \ref{sectInclusion}. Finally, in Section \ref{sectA} we prove these conjectures in type $A$,  and we prove finer properties specific to the case when $\g=\mathfrak{sl}_2$.

 \vspace{0.4cm}

\textit{The author is supported by the European Research Council under the European Union's Framework Programme 
H2020 with ERC Grant Agreement number 647353 Qaffine.}

	\section{Cluster algebras and quantum cluster algebras}\label{sectclusterqcluster}

Cluster algebras were defined by Fomin and Zelevinsky in the 2000's in a series of fundamental papers \citep{CA1},\citep{CA2},\citep{CA3} and \citep{CA4}. They were first introduced to study total positivity and canonical bases in quantum groups but soon applications to many different fields of mathematics were found.

In \citep{qCA}, Berenstein and Zelevinsky introduced natural non-commutative deformations of cluster algebras called \emph{quantum cluster algebras}. 

In this section, we recall the definitions of these objects. The interested reader may refer to the aforementioned papers for more details, or to reviews, such as \citep{CAnotes}.

		\subsection{Cluster algebras}\label{sectcluster}

Let $m\geq n$ be two positive integers and let $\F$ be the field of rational functions over $\mathbb{Q}$ in $m$ independent commuting variables. Fix of subset $\textbf{ex} \subset \llbracket 1,m\rrbracket$ of cardinal $n$.

\begin{defi}
A \emph{seed} in $\F$ is a pair $(\tilde{\pmb x}, \tilde{B})$, where
\begin{itemize}
	\item[•] $\tilde{\pmb x}=\{ x_1,\ldots,x_m\}$ is an algebraically independent subset of $\F$ which generates $\F$.
	\item[•] $\tilde{B}=(b_{i,j})$ of $\tilde{B}$ is a $m\times n$ integer matrix with rows labeled by $\llbracket 1,m\rrbracket$ and columns labeled by $\textbf{ex}$ such that
	\begin{enumerate}
		\item the $n\times n$ submatrix $B=(b_{ij})_{i,j\in \textbf{ex}}$ is skew-symmetrizable.
		\item $\tilde{B}$ has full rank $n$.
	\end{enumerate}	 
\end{itemize}

The matrix $B$ is called the \emph{principal part} of $\tilde{B}$, $\pmb x = \{x_j \mid j\in \textbf{ex} \} \subset \tilde{\pmb x}$ is the \emph{cluster} of the seed $(\tilde{\pmb x}, \tilde{B})$, $\textbf{ex}$ are the \emph{exchangeable indices} and $\pmb c = \tilde{\pmb x} \smallsetminus \pmb x$ is the set of \emph{frozen variables}.

For all $k\in\textbf{ex}$, define the \emph{seed mutation} in direction $k$ as the transformation from $(\tilde{\pmb x}, \tilde{B})$ to $\mu_k(\tilde{\pmb x}, \tilde{B}) = (\tilde{\pmb x'}, \tilde{B'})$, with
\begin{itemize}
	\item[•] $\tilde{B}'=\mu_k(\tilde{B})$ is the $m\times n$ matrix whose entries are given by
	\begin{equation}
b_{ij}'= \left\lbrace \begin{array}{ll}
				-b_{ij} & \text{ if } i=k \text{ or } j=k, \\
				b_{ij} + \frac{|b_{ik}|b_{kj}+ b_{ik}|b_{kj}|}{2} & \text{ otherwise}.
\end{array}\right.
	\end{equation}
This operation is called \emph{matrix mutation} in direction $k$.
This matrix can also be obtained via the operation
\begin{equation}\label{eqBtilde}
\tilde{B}' = \mu_k(\tilde{B}) = E_k\tilde{B}F_k,
\end{equation}
where $E_k$ is the $m\times m$ matrix with entries
\begin{equation}\label{Ek}
(E_k)_{ij} = \left\lbrace \begin{array}{cl}
	\delta_{ij} & \text{ if } j \neq k, \\
	-1 & \text{ if } j=i=k, \\
	\max(0, -b_{ik}) & \text{ if } i\neq j =k,
\end{array}\right.
\end{equation}
and $F_k$ is the $n\times n$ matrix with entries
\begin{equation}\label{Fk}
(F_k)_{ij} = \left\lbrace \begin{array}{cl}
	\delta_{ij} & \text{ if } i \neq k, \\
	-1 & \text{ if } j=i=k, \\
	\max(0, b_{kj}) & \text{ if } i =k\neq j.
\end{array}\right.
\end{equation}
	\item[•] $\tilde{\pmb x}' = \left(\tilde{\pmb x} \smallsetminus \{x_k\}\right) \cup \{x_k'\}$, where $x_k'\in\F$ is determined by the \emph{exchange relation}
\begin{equation}
x_k x_k'= \prod_{\substack{i\in [1,m] \\ b_{ik}>0}} x_i^{b_{ik}} + \prod_{\substack{i\in [1,m] \\ b_{ik}<0}} x_i^{-b_{ik}}.
\end{equation}
\end{itemize}
\end{defi}
\begin{rem}
$(\tilde{\pmb x}',\tilde{B}')$ is also a seed in $\F$ and the seed mutation operation is involutive: $\mu_k(\tilde{\pmb x}',\tilde{B}')=(\tilde{\pmb x},\tilde{B})$. Thus, we have a equivalence relation: $(\tilde{\pmb x},\tilde{B})$ is \emph{mutation-equivalent} to $(\tilde{\pmb x}',\tilde{B}')$, denoted by $(\tilde{\pmb x},\tilde{B})\sim(\tilde{\pmb x}',\tilde{B}')$, if $(\tilde{\pmb x}',\tilde{B}')$ can be obtained from $(\tilde{\pmb x},\tilde{B})$ by a finite sequence of seed mutations.
\end{rem}

Graphically, if the matrix $\tilde{B}$ is skew-symmetric, it can be represented by a quiver and the matrix mutation by a simple operation on the quiver. Fix $\tilde{B}$ a skew-symmetric matrix. Define the quiver $Q$ whose set of vertices is $\llbracket 1,m\rrbracket$, where the vertices corresponding to $\pmb c$ are usually denoted by a square $\Square$ and called \emph{frozen vertices}. For all $i\in \llbracket 1,m\rrbracket$, $j\in\textbf{ex}$, $b_{ij}$ is the number of arrows from $i$ to $j$ (can be negative if the arrows are from $j$ to $i$).
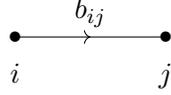
\[
\begin{tikzpicture}
\draw (0,0) node[label=below:{$i$}]{$\bullet$} ;
\draw (2,0) node[label=below:{$j$}]{$\bullet$};
\draw[->-] (0,0) to (2,0) ;
\draw (1,0) node[above]{$b_{ij}$};
\end{tikzpicture}
\]
In this context, the operation of matrix mutation can be translated naturally to an operation on the quiver $Q$. For $k\in\textbf{ex}$, the quiver $Q'=\mu_k(Q)$ is obtained from $Q$ by the following operations:
\begin{itemize}
	\item[•] For each pair of arrows $i \to k \to j$ in $Q$, create an arrow from $i$ to $j$.
	\item[•] Invert all arrows adjacent to $k$.
	\item[•] Remove all 2-cycles that were possibly created.
\end{itemize}

\begin{defi}
Let $\mathcal{S}$ be a mutation-equivalence class of seeds in $\F$. The \emph{cluster algebra} $\mathcal{A}(\mathcal{S})$ associated to $\mathcal{S}$ is the $\Z[\pmb c^\pm]$-subalgebra of $\F$ generated by all the clusters of all the seeds in $\mathcal{S}$.
\end{defi}

	\subsection{Compatible pairs}\label{sectcompair}

A quantum cluster algebra is a non-commutative version of a cluster algebra. Cluster variables will not commute anymore, but, if they are in the same cluster, commute up to some power of an indeterminate $t$. These powers can be encoded in a skew-symmetric matrix $\Lambda$. In order for the quantum cluster algebra to be well-defined, one needs to check that these $t$-commutation relations behave well with the exchange relations. This is made explicit via the notion of compatible pairs.

\begin{defi}\label{defcompair}
Let $\tilde{B}$ be a $m\times n$ integer matrix, with rows labeled by $\llbracket 1,m\rrbracket$ and columns labeled by $\textbf{ex}$. Let $\Lambda = (\lambda_{ij})_{1\leq i,j\leq m}$ be a skew-symmetric $m\times m$ integer matrix. We say that $(\Lambda,\tilde{B})$ forms a \emph{compatible pair} if, for all $i\in \textbf{ex}$ and $1\leq j \leq m$, we have
\begin{equation}\label{compair}
\sum_{k=1}^m b_{ki}\lambda_{kj} = \delta_{i,j}d_i,
\end{equation}
with $(d_i)_{i\in\textbf{ex}}$ some positive integers. Relation (\ref{compair}) is equivalent to saying that, up to reordering, if $\textbf{ex}= \llbracket 1,n\rrbracket$, the matrix $\tilde{B}^T\Lambda$ consists of two blocks, a diagonal $n\times n$ block, and a $n\times (m-n)$ zero block:
\begin{equation*}
\left( \begin{array}{cccc|cccc}
d_1 & & & & & & & \\
& d_2 & & & & & & \\
& &  \ddots & & & (0) & \\
& & & d_n & & & &
\end{array} \right)
\end{equation*}
\end{defi}

Fix a compatible pair $(\Lambda,\tilde{B})$ and fix $k\in \textbf{ex}$. Define, in a similar way as in (\ref{eqBtilde}),
\begin{equation}
\Lambda' = \mu_k(\Lambda) : = E_k^T\Lambda E_k,
\end{equation}
with $E_k$ from (\ref{Ek}).

\begin{prop}[\citep{qCA}]
The pair $(\Lambda', \tilde{B}')$ is compatible.
\end{prop} 

We say that $(\Lambda', \tilde{B}')$ is the \emph{mutation} in direction $k$ of the pair $(\Lambda, \tilde{B})$, and we will use the notation:
\begin{equation}
\mu_k(\Lambda, \tilde{B})  := \left(\mu_k(\Lambda), \mu_k(\tilde{B})\right) = (\Lambda', \tilde{B}').
\end{equation}

\begin{prop}[\citep{qCA}]
The mutation of a compatible pair is involutive. For any compatible pair $(\Lambda, \tilde{B})$ and any mutation direction $k\in \textbf{ex}$, $\mu_k(\mu_k(\Lambda, \tilde{B})) = (\Lambda, \tilde{B})$.
\end{prop}

	\subsection{Definition of quantum cluster algebras}\label{sectqcluster}

We now introduce the last notions we need in order to define quantum cluster algebras.

Let $t$ be a formal variable. Consider $\Z[t^{\pm 1/2}]$, the ring of Laurent polynomials in the variable $t^{1/2}$. 

Recall that any skew-symmetric integer matrix $\Lambda$ of size $m\times m$ determines a skew-symmetric $\Z$-bilinear form on $\Z^m$, which will also be denoted by $\Lambda$:
\begin{equation}
\Lambda(\e_i,\e_j) := \lambda_{i,j}, \quad \forall i,j \in\llbracket 1,m \rrbracket,
\end{equation}
where $\{\e_i\mid 1 \leq i \leq m\}$ is the standard basis of $\Z^m$.

\begin{defi}
The \emph{quantum torus} $\mathcal{T}=(\mathcal{T}(\Lambda),\ast)$ associated with the skew-symmetric bilinear form $\Lambda$ is the  $\Z[t^{\pm 1/2}]$-algebra generated by the $\{ X^\e \mid \e \in \Z^m \}$, together with the $t$-commuting relations:
\begin{equation}
X^\e\ast X^\f = t^{\Lambda(\e,\f)/2}X^{\e+\f} = t^{\Lambda(\e,\f)}X^\f\ast X^\e, \quad \forall \e,\f\in \Z^m.
\end{equation} 
\end{defi}

The quantum torus $\mathcal{T}(\Lambda)$ is an Ore domain (see details in \citep{qCA}), thus it is contained in its skew-field a fractions $\F=(\F,\ast)$. The field $\F$ is a $\Q(t^{1/2})$-algebra.

\begin{defi}
A \emph{toric frame} in $\F$ is a map $M: \Z^m \to \F\setminus \{0\}$ of the form
\begin{equation}
M(c) = \phi(X^{\eta(c)}), \quad\forall c \in\Z^m,
\end{equation}
where $\phi: \F \to \F$ is a $\Q(t^{1/2})$-algebra automorphism and $\eta: \Z^m\to \Z^m$ is an isomorphism of $\Z$-modules.
\end{defi}

For any toric frame $M$, define $\Lambda_M : \Z^m\times \Z^m\to \Z$, a skew-symmetric bilinear form, by
\begin{equation}
\Lambda_M(\e,\f)= \Lambda(\eta(\e),\eta(\f)), \quad\forall\e,\f\in \Z^m.
\end{equation}

Then,
\begin{equation}
M(\e)\ast M(\f)= t^{\Lambda_M(\e,\f)/2}M(\e+\f)=t^{\Lambda_M(\e,\f)}M(\f)\ast M(\e). 
\end{equation}

\begin{defi}
A \emph{quantum seed} in $\F$ is a pair $(M,\tilde{B})$, where
\begin{itemize}
	\item[•] $M$ is a toric frame in $\F$,
	\item[•] $\tilde{B}$ is an $m \times \textbf{ex}$ integer matrix,
	\item[•] the pair $(\Lambda_M,\tilde{B})$ is compatible, as in Definition \ref{defcompair}.
\end{itemize}
\end{defi}

Next, we need to define mutations of quantum seeds. Let $(M,\tilde{B})$ be a quantum seed, and fix $k\in\textbf{ex}$. Define $M':\Z^m \to \F\setminus \{0\}$ by setting
\begin{equation*}
M'(\f)   = \left\lbrace \begin{array}{ll}
\sum_{p=0}^k\binom{f_k}{p}_{t^{d_k/2}}M(E_k\f + pb^k) & \text{if } f_k\geq 0,\\
 M'(-\f)^{-1} & \text{otherwise},
\end{array}\right.
\end{equation*}
where $E_k$ is the matrix from (\ref{Ek}), $b^k\in\Z^m$ is the $k$th column of $\tilde{B}$ and the $t$-binomial coefficient is defined by
\begin{equation}
\binom{r}{p}_t:=\frac{(t^r-t^{-r})(t^{r-1}-t^{-r+1})\cdots (t^{r-p+1})-t^{-r+p-1}}{(t^p-t^{-p})(t^{p-1}-t^{-p+1})\cdots (t-t^{-1})}, \quad \forall 0\leq p\leq r.
\end{equation}
Recall the definition of the mutated matrix $\tilde{B}'=\mu_k(\tilde{B})$ from Section \ref{sectcluster}. Then the \emph{mutation} in direction $k$ of the quantum seed $(M,\tilde{B})$ is the pair $\mu_k(M,\tilde{B})=(M',\tilde{B}')$

\begin{prop}[\citep{qCA}]
\begin{enumerate}[(1)]
	\item The pair $(M',\tilde{B}')$ is a quantum seed.
	\item The mutation in direction $k$ of the compatible pair $(\Lambda_M,\tilde{B})$ is the pair $(\Lambda_{M'},\tilde{B}')$.
\end{enumerate}
\end{prop}

For a quantum seed $(M,\tilde{B})$, let $\tilde{\textbf{X}} = \{X_1,\ldots,X_m\}$ be the free generating set of $\F$, given by $X_i:=M(\e_i)$. Let $\textbf{X} =\{ X_i \mid i \in \textbf{ex} \}$, we call it the \emph{cluster} of the quantum seed $(M,\tilde{B})$, and let $\textbf{C}=\tilde{\textbf{X}}\setminus \textbf{X}$. 

For all $k\in\textbf{ex}$, if $(M',\tilde{B}')=\mu_k(M,\tilde{B})$, then the $X_i'=M'(\e_i)$ are obtained by:
\begin{equation}\label{qmut}
X_i' =\left\lbrace
\begin{array}{ll}
 X_i  & \text{ if } i \neq k, \\
 M\left( - \e_k + \sum_{b_{ik}>0}b_{ik}\e_i\right) +M\left( - \e_k - \sum_{b_{ik}<0}b_{ik}\e_i\right) & \text{ if } i = k.
\end{array}\right.
\end{equation}

The mutation of quantum seeds, as the mutation of compatible pairs, is an involutive process: $\mu_k(M',\tilde{B}')=(M,\tilde{B})$. Thus, as before, we have an equivalence relation: two quantum seeds $(M_1,\tilde{B}_1)$ and $(M_2,\tilde{B}_2)$ are \emph{mutation equivalent} if $(M_2,\tilde{B}_2)$ can be obtained from $(M_1,\tilde{B}_1)$ by a sequence of quantum seed mutations. From (\ref{qmut}), the set $\textbf{C}$ only depends on the mutation equivalence class of the quantum seed. The variables in $\textbf{C}$, $(X_i)_{i\notin\textbf{ex}}$, are called the \emph{frozen variable} of the mutation equivalence class.

\begin{defi}
Let $\mathcal{S}$ be a mutation equivalence class of quantum seeds in $\F$ and $\textbf{C}$ the set of its frozen variables. The \emph{quantum cluster algebra} $\mathcal{A}(\mathcal{S})$ associated with $\mathcal{S}$ is the $\Z[t^{1/2}]$-subalgebra of the skew-field $\F$ generated by the union of all clusters in all seeds in $\mathcal{S}$, together with the elements of $\textbf{C}$ and their inverses.
\end{defi}

		\subsection{Laurent phenomenon and quantum Laurent phenomenon}\label{sectLP}

One of the main properties of cluster algebras is the so-called \emph{Laurent Phenomenon} which was formulated in \citep{CA3}. Quantum cluster algebras present a counterpart to this result called the \emph{Quantum Laurent Phenomenon}.

Here, we follow \citep[Section 5]{qCA}. In order to state this result, one needs the notion of \emph{upper cluster algebras}.

Fix $(M,\tilde{B})$ a quantum seed, and $\tilde{\textbf{X}}=\{ X_1,\ldots,X_m\}$ given by $X_k=M(e_k)$. Let $\Z\mathbb{P}[\textbf{X}^{\pm 1}]$ denote the based quantum torus generated by the $(X_k)_{1\leq k\leq m}$; it is a $\Z[t^{\pm 1/2}]$-subalgebra of $\F$ with basis $\{ M(c) \mid c\in \Z^m \}$, such that the ground ring $\Z\mathbb{P}$ is the ring of integer Laurent polynomials in the variables $t^{1/2}$ and $(X_j)_{j\notin \textbf{ex}}$. For $k\in \textbf{ex}$, let $(M_k,\tilde{B}_k)$ be the quantum seed obtained from $(M,\tilde{B})$ by mutation in direction $k$, and let $\textbf{X}_k$ denote its cluster, thus:
\begin{equation*}
\textbf{X}_k = \left(\textbf{X} \setminus \{X_k\}\right)\cup\{X_k'\}.
\end{equation*}
Define the quantum upper cluster algebra as the $\Z\mathbb{P}$-subalgebra of $\F$ given by
\begin{equation}
\mathcal{U}(M,\tilde{B}) := \Z\mathbb{P}[\textbf{X}^{\pm 1}] \cap \bigcap_{k\in\textbf{ex}} \Z\mathbb{P}[\textbf{X}^{\pm 1 }_k].
\end{equation}

\begin{theo}\citep[Theorem 5.1]{qCA}\label{theoqphen}
The quantum upper algebra $\mathcal{U}(M,\tilde{B})$ depends only on the mutation-equivalence class of the quantum seed $(M,\tilde{B})$.
\end{theo}	
Thus we use the notation: $\mathcal{U}(M,\tilde{B})=\mathcal{U}(\mathcal{S})$, where $\mathcal{S}$ is the mutation-equivalence class of $(M,\tilde{B})$, one has:
\begin{equation}
\mathcal{U}(\mathcal{S}) = \bigcap_{(M,\tilde{B})\in \mathcal{S}} \Z\mathbb{P}[\textbf{X}^{\pm 1}].
\end{equation}
Theorem \ref{theoqphen} has the following important corollary, which we refer to as the \emph{quantum Laurent phenomenon}.
\begin{cor}\citep[Corollary 5.2]{qCA}\label{propLaurentP}
The cluster algebra $\mathcal{A}(\mathcal{S})$ is contained in $\mathcal{U}(\mathcal{S})$. Equivalently, $\mathcal{A}(\mathcal{S})$ is contained in the quantum torus $\Z\mathbb{P}[\textbf{X}^{\pm 1}]$ for every quantum seed $(M,\tilde{B})\in \mathcal{S}$ of cluster $\textbf{X}$.
\end{cor}

		\subsection{Specializations of quantum cluster algebras}\label{sectspec}

%The question of specialization of quantum cluster algebras into (classical) cluster algebras is not trivial, see \citep{QCAS} for a detailed review and some results on the subject. Nonetheless, we state here some results we will need later on. 

Fix a quantum seed $(M,\tilde{B})$ and $\textbf{X}$ its cluster. The based quantum torus $\Z\mathbb{P}[\textbf{X}^{\pm 1}]$ specializes naturally at $t=1$, via the ring morphism:
\begin{equation}\label{eval1}
\pi : \Z\mathbb{P}[\textbf{X}^{\pm 1}]\to \Z[\tilde{\textbf{X}}^{\pm 1}],
\end{equation}
	such that
\begin{equation*}
	\begin{array}{ll}
	\pi(X_k)  & = X_k,\quad (1\leq k\leq m)\\
	\pi(t^{\pm 1/2})  &= 1.
	\end{array}
\end{equation*}

If we restrict this morphism to the quantum cluster algebra $\mathcal{A}(\mathcal{S})$, it is not clear that we recover the (classical) cluster algebra $\mathcal{A}(\tilde{B})$. This question was tackled in a recent paper by Geiss, Leclerc and Schröer \citep{QCAS}.

\begin{rem}\label{remcorresp}
For a combinatorial point of view, the cluster algebras $\A(\mathcal{S})$ and $\A(\tilde{B})$ are constructed on the same quiver $\tilde{B}$, and the mutations have the same effect on the quiver. Assume the initial seeds are fixed and identified, via the morphism (\ref{eval1}). Then, each quantum cluster variable in $\A(\mathcal{S})$ is identified to a cluster variable in $\A(\tilde{B})$. 
\end{rem}

\begin{prop}\citep[Lemma 3.3]{QCAS}\label{propeval}
The restriction of $ \pi$ to $\mathcal{A}(\mathcal{S})$ is surjective on $\mathcal{A}(\tilde{B})$, and quantum cluster variables are sent to the corresponding cluster variables.
\end{prop}

They also conjectured that the specialization at $t=1$ of the quantum cluster algebra is isomorphic to the classical cluster algebra, and gave a proof under some assumptions on the initial seed. 

Nevertheless, by applying Proposition \ref{propeval} to different seeds (while keeping the identification (\ref{eval1}) of the initial seeds), one gets:
\begin{cor}\label{coreval}
The evaluation morphism $\pi$ sends all quantum cluster monomials to the corresponding cluster monomials.
\end{cor}

%%%%%%%%%%%%%%%%%%%%%%%%%%%%%%%%%%%%%%%%%%%%%%%%%%%%%%%%%%%%%%%%%%%%%%%%% a préciser plus tard ? %%%%%%%%%%%%%%%%%%%%%%%%%%%%%%%%%%%%%%%%%%%%%%%%%%%%%%%%%%%%%%%%%%%%%%%%%%%%%%%%%%%%%%%%
	
		\subsection{Positivity}\label{sectpos}

Let us state a last general result on quantum cluster algebras: Davison's positivity theorem \citep{PQCA}.

We have recalled in Section \ref{sectLP} that each (quantum) cluster variable can be written as a Laurent polynomial in the initial (quantum) cluster variables (and $t^{1/2}$). For classical cluster algebras, Fomin-Zelevinski conjectured that these Laurent polynomials have positive coefficients. The so-called \emph{positivity conjecture} was proven by Lee-Schiffler in \citep{PCA}. 

For quantum cluster algebras, the result is the following.

\begin{theo}\citep[Theorem 2.4]{PQCA}\label{theopos}
Let $\mathcal{A}$ be a quantum cluster algebra defined by a compatible pair $(\Lambda,\tilde{B})$. For a mutated toric frame $M'$ and a quantum cluster monomial $Y$, let us write:
\begin{equation}
Y= \sum_{\e \in\Z^m}a_{\e}(t^{1/2})M'(\e),
\end{equation}
with $a_{\e}(t^{1/2}) \in \Z[t^{\pm 1/2}]$. Then the coefficients $a_{\e}(t^{1/2})$ have positive coefficients. 

Moreover, they can be written with in the form
\begin{equation}
a_{\e}(t^{1/2})= t^{-\deg(b_\e(t))/2}b_{\e}(t),
\end{equation}
where $b_{\e}(t) \in \N[q]$, i.e. each polynomial $a_{\e}(t^{1/2})$ contains only even or odd powers of $t^{1/2}$.
\end{theo}

	\section{Cartan data and quantum Cartan data}\label{sectCartan}

We fix here some notations for the rest of the paper.

		\subsection{Root data}\label{sectdata}	
		
Let the $\g$ be a simple Lie algebra of type $A$,$D$ or $E$ of rank $n$, and let $I:=\{ 1, \ldots, n \}$. 

The \emph{Cartan matrix} of $\g$ is the $n\times n$ matrix $C$ such that 
\begin{equation*}
C_{i,j}=\left\lbrace \begin{array}{rl}
				2 & \text{ if }  i=j,\\
				-1 & \text{ if }  i \sim j \quad\text{( } i \text{ and } j \text{ are adjacent vertices of } \gamma  \text{ ) },\\
				0 & \text{ otherwise.}
\end{array}\right.
\end{equation*}

Let us denote by $(\alpha_{i})_{i\in I}$ the simple roots of $\mathfrak{g}$, $(\alpha_{i}^{\vee})_{i\in I}$ the simple coroots and $(\omega_{i})_{i\in I}$ the fundamental weights. We will use the usual lattices $Q= \bigoplus_{i\in I}\mathbb{Z}\alpha_{i}$, $Q^{+}= \bigoplus_{i\in I}\mathbb{N}\alpha_{i}$ and $P= \bigoplus_{i\in I}\mathbb{Z}\omega_{i}$. Let $P_{\mathbb{Q}}=P\otimes \mathbb{Q}$, endowed with the partial ordering : $\omega \leq \omega'$ if and only if $\omega' - \omega \in Q^{+}$. 

The Dynkin diagram of $\g$ is numbered as in \citep{IDLA}, and let $a_1,a_2,\ldots,a_n$ be the Kac labels ($a_0=1$).

		\subsection{Quantum Cartan matrix}\label{sectQCM}
		
Let $z$ be an indeterminate.

\begin{defi}
The \textit{quantum Cartan matrix} of $\g$ is the matrix $C(z)$ with entries,
\begin{equation*}
C_{ij}(z)=\left\lbrace \begin{array}{cl}
				z+z^{-1} & \text{ if } i=j,\\
				-1 & \text{ if } i \sim j,\\
				0 & \text{ otherwise.}
\end{array}\right.
\end{equation*}
\end{defi}
\begin{rem}
The evaluation $C(1)$ is the Cartan matrix of $\mathfrak{g}$. As $\det(C)\neq 0$, then $\det(C(z))\neq 0$ and we can define $\tilde{C}(z)$, the inverse of the matrix $C(z)$. The entries of the matrix $\tilde{C}(z)$ belong to $\mathbb{Q}(z)$. 
\end{rem}

One can write 
\begin{equation*}
C(z) = (z+z^{-1})\id - A,
\end{equation*}
where $A$ is the adjacency matrix of $\gamma$. Hence,
\begin{equation*}
\tilde{C}(z) = \sum_{m=0}^{+\infty}(z+z^{-1})^{-m-1}A^{m}.
\end{equation*}
Therefore, we can write the entries of $\tilde{C}(z)$ as power series in $z$. For all $i,j\in I$,
\begin{equation}\label{qCm}
\tilde{C}_{i j}(z) = \sum_{m=1}^{+\infty}\tilde{C}_{i,j}(m)z^{m}\quad \in \mathbb{Z}[[z]].
\end{equation}
\begin{ex}\begin{enumerate}[(i)]\label{exsl2-1}
	\item For $\g=\mathfrak{sl}_{2}$, one has 
\begin{equation}
\tilde{C}_{1 1} = \sum_{n=0}^{+\infty}(-1)^{n}z^{2n+1} = z-z^{3}+z^{5}-z^{7}+z^{9}-z^{11} + \cdots \\
\end{equation}
		\item For $\g=\mathfrak{sl}_{3}$, one has 
\begin{align*}
\tilde{C}_{ii} &= z-z^{5}+z^{7}-z^{11}+z^{13} + \cdots, \quad 1 \leq i \leq 2\\
\tilde{C}_{ij} &= z^{2}-z^{4}+z^{8}-z^{10}+z^{14} + \cdots, \quad 1 \leq i \neq j \leq 2.
\end{align*}
\end{enumerate}
\end{ex}

We will need the following lemma:
\begin{lem}\label{lemCtildeC}
For all $(i,j)\in I^2$, 
\begin{align*}
\tilde{C}_{ij}(m-1) + \tilde{C}_{ij}(m+1) - \sum_{k \sim j}\tilde{C}_{ik}(m) & = 0, \quad \forall m \geq 1, \\
\tilde{C}_{ij}(1) & = \delta_{i,j}.
\end{align*}
\end{lem}

\begin{proof}
By definition of $\tilde{C}$, one has
\begin{equation}
\tilde{C}(z)\cdot C(z) = \id \quad \in \mathcal{M}_n(\Q(z)).
\end{equation}
By writing $\tilde{C}(z)$ as a formal power series, and using the definition of $C(z)$, we obtain, for all $(i,j)\in I^2$, 
\begin{equation}
\sum_{m=0}^{+\infty}\left( \tilde{C}_{ij}(m)(z^{m+1}+z^{m-1}) - \sum_{k\sim j}\tilde{C}_{ik}(m)z^m \right) = \delta_{i,j} \quad \in \mathbb{C}[[z].
\end{equation}
Which is equivalent to
\begin{align*}
\tilde{C}_{ij}(m-1) + \tilde{C}_{ij}(m+1) - \sum_{k \sim j}\tilde{C}_{ik}(m) & = 0, \quad \forall m \geq 1, \\
\tilde{C}_{ij}(1) - \sum_{k\sim j}\tilde{C}_{ij}(0) & = \delta_{i,j}, \\
\tilde{C}_{ij}(0) & = 0.
\end{align*}
\end{proof}

		\subsection{Infinite quiver}\label{sectinftyquiver}

Next, let us define an infinite quiver $\Gamma$ as in \citep{ACAA}. Let $\tilde{\Gamma}$ be the quiver with vertex set $I\times \mathbb{Z}$ and arrows
\begin{equation}
\left( (i,r) \rightarrow (j,s) \right) \Longleftrightarrow \left( C_{i,j} \neq 0 \text{ and } s= r + C_{i,j} \right).
\end{equation}

This quiver has two isomorphic connected components (see \citep{ACAA}). Let $\Gamma$ be one of them, and the $\hat{I}$ be its set of vertices.
 
\begin{ex}\label{exsl4}
For $\g=\mathfrak{sl}_{4}$, fix $\hat{I}$ to be 
\begin{equation*}
\hat{I} := \left\lbrace (1,2p)\mid p\in\mathbb{Z} \right\rbrace \cup \left\lbrace (2,2p+1)\mid p\in\mathbb{Z} \right\rbrace \cup\left\lbrace (3,2p)\mid p\in\mathbb{Z} \right\rbrace,
\end{equation*}
and $\Gamma$ is the following:
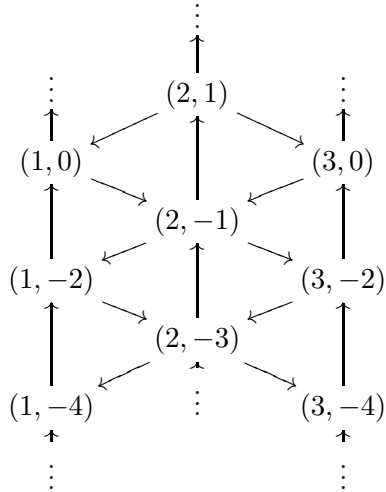
\[
\xymatrix@R=0.5em@C=1.5em{
	& \vdots & \\
\vdots & (2,1) \ar[dl]\ar[dr]\ar[u] & \vdots \\
(1,0)\ar[dr]\ar[u] & & (3,0) \ar[dl]\ar[u] \\
& (2,-1) \ar[dl]\ar[dr]\ar[uu] & \\
(1,-2)\ar[dr]\ar[uu] & & (3,-2)\ar[dl]\ar[uu] \\
& (2,-3) \ar[dl]\ar[dr]\ar[uu] & \\
(1,-4)\ar[uu] &  \vdots \ar[u] & (3,-4)\ar[uu] \\
\vdots \ar[u] &   	& \vdots \ar[u]
}
\]
\end{ex}

	\section{Category $\cO$ of representations of quantum loop algebras}\label{sectO}

We now start with the more representation theoric notions of this paper. We first recall the definitions of the quantum loop algebra and its Borel subalgebra, before introducing the Hernandez-Jimbo category $\cO$ of representations, as well as some known results on the subject. We will sporadically use concepts and notations from the two previous sections. 

			\subsection{Quantum loop algebra and Borel subalgebra}

%Now consider the affine Lie algebra $\hat{\mathfrak{g}}$ obtained from $\mathfrak{g}$.

Fix a nonzero complex number $q$, which is not a root of unity, and $h\in\mathbb{C}$ such that $q=e^{h}$. Then for all $r\in \mathbb{Q}, q^{r}:=e^{rh}$ is well-defined. Since $q$ is not a root of unity, for $r,s\in \mathbb{Q}$, we have $q^{r}=q^{s}$ if and only if $r=s$. 

We will use the following standard notations.
\begin{equation*}
\begin{array}{ccc}
[m]_{z}=\frac{z^{m}-z^{-m}}{z-z^{-1}},& [m]_{z}! = \prod_{j = 1}^{m}[j]_{z}, & \genfrac[]{0pt}{0}{r}{s}_{z} = \frac{[r]_{z}!}{[s]_{z}![r-s]_{z}!}
\end{array}
\end{equation*}

\begin{defi}
One defines the \textit{quantum loop algebra} $\Uqg$ as the $\mathbb{C}$-algebra generated by $e_{i},f_{i},k_{i}^{\pm 1}, 0\leq i \leq n$, together with the following relations, for $0\leq i,j\leq n$ :
\begin{equation}
\begin{gathered}
k_{i}k_{j}=k_{j}k_{i}, \quad k_i k_i^{-1}=k_i^{-1}k_i=1,\quad k_0^{a_0}k_1^{a_1}\cdots k_n^{a_n}=1,\\
\left[ e_{i},f_{j} \right]= \delta_{i,j}\frac{k_{i}-k_{i}^{-1}}{q-q^{-1}},\\
k_{i}e_{j}k_{i}^{-1}=q^{C_{ij}}e_{j}, \quad k_{i}f_{j}k_{i}^{-1}=q^{-C_{ij}}e_{j}, \\
\sum_{r=0}^{1-C_{ij}}(-1)^{r} e_{i}^{(1-C_{ij}-r)}e_{j}e_{i}^{(r)}  =0,~ (i\neq j), \\
\sum_{r=0}^{1-C_{ij}}(-1)^{r}f_{i}^{(1-C_{ij}-r)}f_{j}f_{i}^{(r)} =0,~ (i\neq j),
\end{gathered}
\end{equation}
where $x_{i}^{(r)} = x_{i}^{r}/[r]_{q}!, (x_{i}=e_{i},f_{i})$.
\end{defi}

%The algebra $\Uqg$ has another presentation, with the \textit{Drinfeld generators} (\citep{ANRY}, \citep{BGA})
%\begin{equation*}
%x_{i,r}^{\pm}(i\in I, r\in \mathbb{Z}), \quad \phi_{i,\pm m}^{\pm} (i\in I, m\geq 0),\quad k_{i}^{\pm 1} (i\in I),
%\end{equation*}
%and some relations we will not recall here, but which are also $q$-deformations of the Weyl and Serre relations.

%Using this presentation, we have a triangular decomposition of $\Uqg$ \citep{BGA}:
%\begin{equation}
%\Uqg \simeq \Uqg^-\otimes \Uqg^0 \otimes \Uqg^+,
%\end{equation}
%where $\Uqg^\pm$ (resp. $\Uqg^0$) is the subalgebra of $\Uqg$ generated by the $(x_{i,r}^\pm)_{i\in I, r\in \mathbb{Z}}$ (resp. $(\phi_{i,\pm m}^\pm)_{i\in I, m\in \mathbb{N}}$).

%The algebra $\Uqg$ has an Hopf algebra structure, where the coproduct and the antipode are given by, for $i\in \{0,\ldots,n\}$,
%\begin{equation}\label{HopfUqg}
%\begin{gathered}
%\Delta(e_{i}) =e_{i}\otimes 1 + k_{i}\otimes e_{i},\\ \Delta(f_{i}) =f_{i}\otimes k_{i}^{-1} + 1\otimes f_{i}, \\ \Delta(k_{i}) = k_{i}\otimes k_{i},\quad S(k_{i}) = k_{i}^{-1} \\
%S(e_{i}) = -k_{i}^{-1}e_{i},\quad S(f_{i}) = -f_{i}k_{i}. 
%\end{gathered}
%\end{equation}

\begin{defi}
The \textit{Borel algebra} $\Uqb$ is the subalgebra of $\Uqg$ generated by the $e_{i}, k_{i}^{\pm 1}$, for $0\leq i\leq n$.
\end{defi}

Both the quantum loop algebra and its Borel subalgebra are Hopf algebras.

From now on, except when explicitly stated otherwise, we are going to consider representations of the Borel algebra $\Uqb$. Particularly, we consider the action of the $\ell$-Cartan subalgebra $\Uqb^0$: a commutative subalgebra of $\Uqb$ generated by the so-called Drinfeld generators:
\begin{equation*}
\Uqb^0 := \left\langle k_{i}^{\pm 1}, \phi_{i,r}^+ \right\rangle_{i\in I, r> 0}.
\end{equation*}

		\subsection{Highest $\ell$-weight modules}

Let $V$ be a $\Uqb$-module and $\omega\in P_{\mathbb{Q}}$ a weight. One defines the \textit{weight space} of $V$ of weight $\omega$ by 
\begin{equation*}
V_{\omega} := \lbrace v\in V \mid k_{i}v=q^{\omega(\alpha_{i}^{\vee})}v, 1\leq i \leq n \rbrace.
\end{equation*}
The vector space $V$ is said to be \textit{Cartan diagonalizable} if $V=\bigoplus_{\omega\in P_{\mathbb{Q}}}V_{\omega}$.
\begin{defi}
A series $\pmb\Psi = (\psi_{i,m})_{i\in I,m\geq 0}$ of complex numbers, such that $\psi_{i,m}\in q^{\mathbb{Q}}$ for all $i\in I$ is called an $\ell$\textit{-weight}. The set of $\ell$-weights is denoted by $P_{\ell}$. One identifies the $\ell$-weight $\pmb\Psi$ to its generating series :
\begin{equation*}
\begin{array}{cc}
\pmb\Psi = (\psi_{i}(z))_{i\in I}, & \psi_{i}(z)=\sum_{m\geq 0}\psi_{i,m}z^{m}.
\end{array}
\end{equation*}
\end{defi}
Let us define some particular $\ell$-weights which are important in our context.

For $\omega \in P_\Q$, let $[\omega]$ be defined as
\begin{equation}
([\omega])_i(z) = q^{\omega(\alpha_i^{\vee})}, 1\leq i \leq n.
\end{equation}

For $i\in I$ and $a\in \mathbb{C}^\times$, let
\begin{itemize}
	\item[•] $\pmb\Psi_{i,a}$ be defined as
\begin{equation}\label{Psi}
\left(\pmb\Psi_{i,a}\right)_j(z) = \left\lbrace\begin{array}{ll}
										1-az & \text{ if } j=i \\
										1 & \text{ if } j\neq i
							\end{array}\right. .
\end{equation}
	\item[•] $Y_{i,a}$ be defined as
\begin{equation}\label{Y}
\left(Y_{i,a}\right)_j(z) = \left\lbrace\begin{array}{ll}
										q\frac{1-aq^{-1}z}{1-aqz} & \text{ if } j=i \\
										1 & \text{ if } j\neq i
							\end{array}\right. .
\end{equation}
\end{itemize}
The sets $P_{\mathbb{Q}}$ and $P_{\ell}$ have group structures (the elements of $P_{\ell}$ are invertible formal series) and one has a surjective group morphism $\varpi : P_{\ell} \to P_{\mathbb{Q}}$ which satisfies $\psi_{i}(0)=q^{\varpi(\pmb\Psi)(\alpha_{i}^{\vee})}$, for all $\pmb\Psi \in P_\ell$ and all $i\in I$.

Let $V$ be  $\Uqb$-module and $\pmb\Psi\in P_{\ell}$ an $\ell$-weight. One defines the \textit{$\ell$-weight space} of $V$ of $\ell$-weight $\pmb\Psi$ by
\begin{equation*}
V_{\pmb\Psi} := \lbrace v \in V \mid \exists p\geq 0,\forall i \in I, \forall m\geq 0, (\phi_{i,m}^{+} - \psi_{i,m})^{p}v = 0 \rbrace.
\end{equation*} 
\begin{rem}
With the usual convention $\phi_{i,0}^{+}=k_{i}$, one has $V_{\pmb\Psi} \subset V_{\varpi(\pmb\Psi)}$.
\end{rem}

\begin{defi}
Let $V$ be a $\Uqb$-module.  It is said to be \textit{of highest} $\ell$\textit{-weight} $\pmb\Psi\in P_{\ell}$ if there is $v\in V$ such that $V=\Uqb v$, 
\begin{equation*}
e_{i}v= 0, \forall i \in I \quad \text{ and } \quad \phi_{i,m}^{+}v=\psi_{i,m}v,~\forall i\in I,m \geq 0.
\end{equation*}
In that case, the $\ell$-weight $\pmb\Psi$ is entirely determined by $V$, it is called the $\ell$-weight of $V$, and $v$ is a highest $\ell$-weight vector of $V$.
\end{defi}

\begin{prop}\citep{ARDRF}
For all $\pmb\Psi\in P_{\ell}$ there is, up to isomorphism, a unique simple highest $\ell$-weight module of $\ell$-weight $\pmb\Psi$, denoted by $L(\pmb\Psi)$.
\end{prop}

\begin{ex}\label{exomega}
For $\omega\in P_{\mathbb{Q}}$, $L([\omega])$ is a one-dimensional representation of weight $\omega$. We also denote it by $[\omega]$ (tensoring by this representation is equivalent to shifting the weights by $\omega$).
\end{ex}

	\subsection{Definition of the category $\mathcal{O}$}

As explained in the Introduction, our focus here is a category $\mathcal{O}$ of representations of the Borel algebra, which was first defined in \citep{ARDRF}, mimicking the usual definition of the BGG category $\cO$ for Kac-Moody algebras. Here, we are going to use the definition in \citep{CABS}, which is slightly different. 

For all $\lambda\in P_{\mathbb{Q}}$, define $D(\lambda) := \{ \omega \in P_{\mathbb{Q}} \mid \omega \leq \lambda \}$.
\begin{defi}\label{defcatO}
 A $\Uqb$-module $V$ is in the category $\mathcal{O}$ if :
\begin{enumerate}
	\item $V$ is Cartan diagonalizable,
	\item For all $\omega \in P_{\mathbb{Q}}$, one has $\dim(V_{\omega}) < \infty$,
	\item\label{3catO} There is a finite number of $\lambda_{1},\ldots,\lambda_{s}\in P_{\mathbb{Q}}$ such that all the weights that appear in $V$ are in the cone $\bigcup_{j=1}^{s}D(\lambda_{j})$.
\end{enumerate}
\end{defi} 

The category $\cO$ is a monoidal category.

\begin{ex}
All finite dimensional $\Uqb$-modules are in the category $\mathcal{O}$.
\end{ex}

Let $\Plr$ be the set of $\ell$-weights $\Psi$ such that, for all $i\in I$, $\Psi_{i}(z)$ is rational. We will use the following result.

\begin{theo}\citep{ARDRF}\label{theoctaO}
Let $\Psi\in P_{\ell}$. Simple objects in the category $\cO$ are highest $\ell$-weight modules. The simple module $L(\Psi)$ is in the category $\cO$ if and only if $\Psi\in \Plr$. Moreover, if $V$ is in the category $\cO$ and $V_{\Psi}\neq 0$, then $\Psi\in \Plr$.
\end{theo}

\begin{ex}\label{exprefund}
For all $i\in I$ and $a\in \mathbb{C}^\times$, define the \emph{prefundamental representations} $L_{i,a}^\pm$ as
\begin{equation}
L_{i,a}^\pm := L(\pmb\Psi_{i,a}^\pm),
\end{equation}
for $\pmb\Psi_{i,a}$ defined in (\ref{Psi}). Then from Theorem \ref{theoctaO}, the prefundamental representations belong to the category $\mathcal{O}$.
\end{ex}

			\subsection{Connection to finite-dimension $\Uqg$-modules}\label{sectfd}
			
Throughout this paper, we will use results already known for finite-dimensional representations of the quantum loop algebra $\Uqg$ with the purpose of generalizing some of them to the context of the category $\cO$ of representations of the Borel subalgebra $\Uqb$. Let us first recognize that this approach is valid.

Let $\mathscr{C}$ be the category of all (type 1) finite-dimensional $\Uqg$-modules.

\begin{prop}\citep{IMqAA}\citep[Proposition 2.7]{FCPLM}
Let $V$ be a simple finite-dimensional $\Uqg$-module. Then $V$ is simple as a $\Uqb$-module.
\end{prop} 

Using this result and the classification of finite-dimensional simple module of quantum loop algebras in \citep{GQG}, one has 

\begin{prop}\label{propdimfinie}
For all $i\in I$, let $P_i(z)\in\mathbb{C}[z]$ be a polynomial with constant term 1. Let $\pmb\Psi=(\Psi_{i})_{i\in I}$ be the $\ell$-weight such that
\begin{equation}
\Psi_{i}(z) = q^{\deg(P_i)}\frac{P_i(zq^{-1})}{P_i(zq)}, \quad \forall i\in I.
\end{equation}
Then $L(\pmb\Psi)$ is finite-dimensional. 

Moreover the action of $\Uqb$ can be uniquely extended to an action of $\Uqg$, and any simple object in the category $\C$ is of this form.
\end{prop}

Hence, the category $\C$ is a subcategory of the category $\cO$ and the inclusion functor preserves simple objects.

\begin{ex}\label{exrepfonda}
For all $i\in I$ and $a\in \mathbb{C}^{\times}$, consider the simple $\Uqb$-module $L(\pmb\Psi)$ of highest $\ell$-weight $Y_{i,a}$, as in (\ref{Y}), then by Proposition \ref{propdimfinie}, $L(Y_{i,a})$ is finite-dimensional. This module is called a \textit{fundamental representation} and will be denoted by
\begin{equation}
V_{i,a}:= L(Y_{i,a}).
\end{equation}
\end{ex}

%More generally, for $i\in I, a\in \mathbb{C}^{\times}$ and $k\geq 1$, define the monomial
%\begin{equation}
%m_{k,a}^{(i)} := Y_{i,a}Y_{i,aq^{2}}\cdots Y_{i,aq^{2k-2}}.
%\end{equation}
%The simple module of highest $\ell$-weight $m_{k,a}^{(i)}$ is denoted by 
%\begin{equation}
%W_{k,a}^{(i)}:=L(m_{k,a}^{(i)}) = L(Y_{i,a}Y_{i,aq^{2}}\cdots Y_{i,aq^{2k-2}}).
%\end{equation} 
%They are called \textit{Kirillov-Reshetikhin modules}, or KR-modules.

%As a consequence of Theorem \ref{theoKC}, the simple modules in $\C$ are indexed by the $\ell$-weights which are of the form
%\begin{equation*}
%[\omega]Y_{i_1,a_1}Y_{i_2,a_2}\cdots Y_{i_m,a_m},
%\end{equation*}
%where $\omega \in P_\Q$ .

In general, simple modules in $\C$ are indexed by monomials in the variables $(Y_{i,a})_{i\in I,a\in \mathbb{C}^\times}$, called \emph{dominant monomials}. Frenkel-Reshetikhin \citep{QCRQAA} defined a $q$\emph{-character morphism} $\chi_q$ (see Section \ref{sectqchar}) on the Grothendieck ring of $\C$. It is an injective ring morphism 
\begin{equation}\label{theoKC}
\chi_q: K_0(\C) \to \hat{\mathcal{Y}}:=\mathbb{Z}[Y_{i,a}^{\pm 1}]_{i\in I, a\in \mathbb{C}^\times}.
\end{equation}

\begin{ex}
In the continuity of Example \ref{exrepfonda}, for $\g=\mathfrak{sl}_2$, one has, for all $a\in\mathbb{C}^\times$,
\begin{equation}
\chi_q(L(Y_{1,a})) = Y_{1,a} + Y_{1,aq^2}^{-1}.
\end{equation}
\end{ex}

			\subsection{Categories $\cO^{\pm}$}
	
Let us now recall the definitions of some subcategories of the category $\cO$, introduced in \citep{CABS}. These categories are interesting to study for different reasons; here we use in particular the cluster algebra structure of their Grothendieck rings.

\begin{defi}\label{defiPosNegPoids}
An $\ell$-weight of $P_{\ell}^{\mathfrak{r}}$ is said to be \textit{positive} (resp. \textit{negative}) if it is a monomial in the following $\ell$-weights :
\begin{itemize}
	\item[•] the $Y_{i,a} = q\pmb\Psi_{i,aq}^{-1}\pmb\Psi_{i,aq^{-1}}$, where $i \in I$ and $a\in \mathbb{C}^{\times}$,
	\item[•] the $\pmb\Psi_{i,a}$ (resp. $\pmb\Psi_{i,a}^{-1}$), where $i \in I$ and $a\in \mathbb{C}^{\times}$,
	\item[•] the $[\omega]$, where $\omega\in P_{\mathbb{Q}}$.
\end{itemize}
\end{defi}

\begin{defi}
The category $\Op$ (resp. $\Om$) is the category of representations in $\cO$ whose simple constituents have a positive (resp. negative) highest $\ell$-weight.
\end{defi}

The category $\Op$ (resp. $\Om$) contains the category of finite-dimensional representations, as well as the positive (resp. negative) prefundamental representations $L_{i,a}^{+}$ (resp. $L_{i,a}^{-}$), defined in (\ref{exprefund}), for all $i\in I, a\in \mathbb{C}^{\times}$.

% Positive $\ell$-weights have a unique factorization into a product of $Y_{i,a}$ and $\pmb\Psi_{i,a}$. In particular, for $\mathfrak{g} = \hat{\mathfrak{sl}_{2}}$, this implies a unique factorization of simple modules into products of prime simple representations in $\Op$ (see \citep[Theorem 7.9]{CABS}).

\begin{theo}\citep{CABS}
The categories $\Op$ and $\Om$ are monoidal categories. 
\end{theo}

%Recall as in \citep[3.6]{QCRQAA} that the Dynkin diagram of $\dot{\mathfrak{g}}$ is a bipartite graph. There is partition $I=I_{0}\cup I_{1}$ of the vertices such that all edges have one endpoint in$I_{0}$ and the other in $I_{1}$. For all $i\in I$, let 
%%\begin{equation}
%\xi_{i} = \left\lbrace \begin{array}{ll}
%							0 & \text{if } i\in I_{0} \\
%							1 & \text{if } i\in I_{1}
%						\end{array} \right. 
%\end{equation}
%The map $i\mapsto \xi_{i}$ is entirely determined by the choice of one $\xi_{i_{0}}\in \lbrace 0,1 \rbrace$, with two possible choices. We fix one of these choices.

		\subsection{The category $\cO_\Z^+$}\label{sectOZpm}

Recall the infinite quiver $\Gamma$ from Section \ref{sectinftyquiver} and its set of vertices $\hat{I}$.

In \citep{CAQAA}, Hernandez and Leclerc defined a subcategory $\C_{\mathbb{Z}}$ of the category $\C$. $\C_{\mathbb{Z}}$ is the full subcategory whose objects satisfy: for all composition factor $L(\pmb\Psi)$, for all $i\in I$, the roots of the polynomials $P_i$, as in Proposition \ref{propdimfinie} are of the form $q^{r+1}$, such that $(i,r)\in \hat{I}$. 

This subcategory is interesting to study because each simple object in $\mathscr{C}$ can be written as a tensor product of simple objects which are essentially in $\mathscr{C}_{\mathbb{Z}}$ (see \citep[Section 3.7]{CAQAA}). Thus, the study of simple modules in $\mathscr{C}$ is equivalent to the study of simple modules in $\mathscr{C}_{\mathbb{Z}}$.

Consider the same type of restriction on the category $\cO$.

\begin{defi}
Let $\cO_{\mathbb{Z}}$ be the subcategory of representations of $\cO$ whose simple components have a highest $\ell$-weight $\pmb\Psi$ such that the roots and poles of $\Psi_{i}(z)$ are of the form $q^{r}$, such that $(i,r)\in \hat{I}$.

We also define $\cO_{\mathbb{Z}}^{\pm}$ as the subcategory of $\cO^{\pm}$ whose simple components have a highest $\ell$-weight $\pmb\Psi$ such that the roots and poles of $\Psi_{i}(z)$ are of the form $q^{r}$, such that $(i,r)\in \hat{I}$.
\end{defi}

		\subsection{The Grothendieck ring $K_0(\cO)$}\label{sectKO}

Hernandez and Leclerc showed that the Grothendieck rings of the categories $\cO^\pm_\mathbb{Z}$ have some interesting cluster algebra structures.

First of all, define $\E$ as the additive group of maps $c : P_\mathbb{Q} \to \mathbb{Z}$ whose support is contained in a finite union of sets of the form $D(\mu)$. For any $\omega\in P_\mathbb{Q}$, define $[\omega] \in \mathcal{E}$ as the $\delta$-function at $\omega$ (this is compatible with the notation in Example \ref{exomega}). The elements of $\mathcal{E}$ can be written as formal sums
\begin{equation}
c = \sum_{\omega\in \supp(c)}c(\omega)[\omega]. 
\end{equation}
$\E$ can be endowed with a ring structure, where the product is defined by 
\begin{equation*}
[\omega] \cdot [\omega'] = [\omega + \omega'], \quad\forall \omega,\omega' \in P_\Q.
\end{equation*}

If $(c_k)_{k\in\mathbb{N}}$ is a countable family of elements of $\mathcal{E}$ such that for any $\omega\in P_\mathbb{Q}$, $c_k(\omega)=0$ except for finitely many $k\in \mathbb{N}$, then $\sum_{k\in\mathbb{N}}c_k$ is a well-defined map from $P_\mathbb{Q}$ to $\mathbb{Z}$. In that case, we say that $\sum_{k\in\mathbb{N}}c_k$ is a \emph{countable sum of elements} in $\mathcal{E}$.

The Grothendieck ring of the category $\cO$  can be viewed as a ring extension of $\mathcal{E}$. Similarly to the case of representations of a simple Lie algebra (see \citep{IDLA}, Section 9.6), the multiplicity of an irreducible representation in a given representation of the category $\cO$ is well-defined. Thus, the Grothendieck ring of the category $\cO$ is formed of formal sums
\begin{equation}
\sum_{\pmb\Psi\in \Plr}\lambda_{\pmb\Psi}[L(\pmb\Psi)],
\end{equation}
such that the $\lambda_{\pmb\Psi}\in \mathbb{Z}$ satisfy:
\begin{equation*}
\sum_{\pmb\Psi\in \Plr,\omega\in P_\mathbb{Q}}|\lambda_{\pmb\Psi}|\dim(L(\pmb\Psi)_\omega)[\omega]\in\mathcal{E}.
\end{equation*}

In this context, $\mathcal{E}$ is identified with the Grothendieck ring of the category of representations of $\cO$ with constant $\ell$-weight. 

A notion of countable sum of elements in $K_0(\cO)$ is defined exactly as for $\mathcal{E}$.

Now consider the cluster algebra  $\mathcal{A}(\Gamma)$ defined by the infinite quiver $\Gamma$ of Section \ref{sectinftyquiver}, with infinite set of coordinates denoted by
\begin{equation}
\pmb z = \left\lbrace z_{i,r} \mid (i,r)\in \hat{I} \right\rbrace.
\end{equation} 

By the \emph{Laurent Phenomenon} (see \ref{propLaurentP}), $\mathcal{A}(\Gamma)$ is contained in $\Z[z_{i,r}^{\pm 1}]_{(i,r)\in \hat{I}}$. Define $\chi : \Z[z_{i,r}^{\pm 1}]\otimes_\Z \E \to \E$, the $\E$-algebra homomorphism by
\begin{equation}\label{chiZ}
\chi(z_{i,r}^{\pm 1}) = \left[ \left( \frac{\mp r}{2}\right) \omega_i \right], \quad ((i,r) \in \hat{I}).
\end{equation}
The map $\chi$ is defined on $ \mathcal{A}(\Gamma)\otimes_\Z \E$, and for each $A\in \mathcal{A}(\Gamma)\otimes_\Z \E$, one can write $\chi(A)=\sum_{\omega\in P_\Q}A_\omega\otimes [\omega]$, and $\abs{\chi}(A)=\sum_{\omega\in P_\Q}\abs{A_\omega}\otimes [\omega]$.

Consider the completed tensor product
\begin{equation}
\mathcal{A}(\Gamma)\hat{\otimes}_\mathbb{Z}\mathcal{E},
\end{equation}
of countable sums $\sum_{k\in \mathbb{N}}A_k$ of elements $A_k\in \mathcal{A}(\Gamma)\otimes_\Z \E$, such that $\sum_{k\in \N}\abs{\chi}(A_k)$ is a countable sum of elements of $\E$, as defined above.

\begin{theo}\citep[Theorem 4.2]{CABS}\label{theoHL}
The category $\cO^+_\mathbb{Z}$ is monoidal, and the identification
\begin{equation}\label{identZL}
z_{i,r}\otimes\left[ \frac{r}{2}\omega_i\right] \equiv [L_{i,q^r}^+], \quad \left( (i,r) \in \hat{I} \right),
\end{equation}
defines an isomorphism of $\mathcal{E}$-algebras
\begin{equation}
\mathcal{A}(\Gamma)\hat{\otimes}_\mathbb{Z}\mathcal{E} \simeq K_0(\cO^+_\mathbb{Z}).
\end{equation}
\end{theo}

\begin{ex}\label{excmut}
We mentioned in the introduction that the Baxter relation (\ref{TQ'}) was an exchange relation for this cluster algebra structure, let us detail this. For $\g=\mathfrak{sl}_2$, the quiver $\Gamma$ is:

\begin{minipage}[c]{2cm}
\begin{tikzpicture}
\node (4) at (0,2){$\vdots$};
\node (2) at (0,1){(1,2)};
\node (0) at (0,0){(1,0)};
\node (-2) at (0,-1){(1,-2)};
\node (-4) at (0,-2){$\vdots$};
\draw [->] (-2) edge (0);
\draw [->] (0) edge (2);
\draw [->] (2) edge (4);
\draw [->] (0,-1.75) -- (0,-1.3);
\end{tikzpicture}
\end{minipage} 
\begin{minipage}{12cm}
If we mutate at the node $(1,0)$, the new cluster variable obtained is 
\[z_{1,0}'=\frac{1}{z_{1,0}}\left( z_{1,2} + z_{1-2} \right).
\]
Thus, via the identification (\ref{identZL}), 
\[z_{1,0}'=\frac{1}{[L_{1,1}^+]}\left( [-\omega_1][L_{1,q^2}^+] + [\omega_1][L_{1,q^{-2}}^+] \right).
\]
We indeed recognize the Baxter relation. 
\end{minipage}

Moreover, the new cluster variable $z_{1,0}'$ identifies to a fundamental representation:
\begin{equation}
z_{1,0}' = [L(Y_{1,q^{-1}})].
\end{equation}

\end{ex}

\begin{rem}
An analog theorem could be written for $K_0(\cO^-_\mathbb{Z})$, as these are isomorphic as $\mathcal{E}$-algebras (\citep[Theorem 5.17]{CABS}).
\end{rem}

		\subsection{The $q$-character morphism}\label{sectqchar}
Here we detail the notion of $q$\emph{-character} on the category $\cO$. This notion extends the $q$-character morphism on the category of finite-dimensional $\Uqg$-modules mentioned in Section \ref{sectfd}.

Similarly to Section \ref{sectKO}, consider $\E_\ell$, the additive group of maps $c :  \Plr \to \Z$ such that the image by $\varpi$ of its support is contained in a finite union of sets of the form $D(\mu)$, and for any $\omega\in P_\mathbb{Q}$, the set $\supp(c)\cap \varpi^{-1}(\{\omega \})$ is finite. The map $\varpi$ extends naturally to a surjective morphism $\varpi : \E_\ell \to \E$. For $\pmb\Psi\in \Plr$, define the delta function $[\pmb\Psi]=\delta_{\pmb\Psi} \in \E_{\ell}$.

The elements of $\mathcal{E}_\ell$ can be written as formal sums
\begin{equation}
c = \sum_{\pmb\Psi\in \Plr}c(\pmb\Psi)[\pmb\Psi]. 
\end{equation}

Endow $\E_\ell$ with a ring structure given by
\begin{equation}\label{eqprodEl}
(c\cdot d)(\pmb\Psi) = \sum_{\pmb\Psi'\pmb\Psi''=\pmb\Psi}c(\pmb\Psi')d(\pmb\Psi''), \quad \left( c,d,\in \E_\ell, \pmb\Psi\in \Plr \right).
\end{equation}
In particular, for $\pmb\Psi,\pmb\Psi' \in \Plr$,
\begin{equation}\label{eqPsiPsi'}
[\pmb\Psi]\cdot[\pmb\Psi']=[\pmb\Psi\pmb\Psi'].
\end{equation}

For $V$ a module in the category $\cO$, define the $q$-character of $V$ as in \citep{QCRQAA}, \citep{ARDRF}:
\begin{equation}
\chi_q(V):= \sum_{\pmb\Psi\in\Plr} \dim(V_{\pmb\Psi})[\pmb\Psi].
\end{equation}

By definition of the category $\cO$, $\chi_q(V)$ is an object of the ring $\E_\ell$.

The following result extends the one from \citep{QCRQAA} to the context of the category $\cO$.

\begin{prop}[\citep{ARDRF}]
The $q$-character map
\begin{equation}
\begin{array}{rrl}
\chi_q: & K_0(\cO) & \to \E_\ell \\
& [V] & \mapsto \chi_q(V),
\end{array}
\end{equation}
is an injective ring morphism.
\end{prop}

\begin{ex}
For any $a\in\mathbb{C}^\times$, $i\in I$, one has \citep{ARDRF,BRSQIM},
\begin{equation}\label{chiqL}
\chi_q(L_{i,a}^+)=\left[ \pmb\Psi_{i,a}\right]\chi_i,
\end{equation}
where $\chi_i = \chi(L_{i,a}^+) \in \E$ does not depend on $a$.

For example, if $\g= \mathfrak{sl}_2$,
\begin{equation}
\chi_1 = \chi = \sum_{r\geq 0}[-2r\omega_1].
\end{equation}
\end{ex}

	\section{Quantum tori}\label{sectQuantumTori}
	
Let $t$ be an indeterminate. The aim of this section is to built a non-commutative quantum torus $\T_t$ which will contain the quantum Grothendieck ring for the category $\cO$.
For the category $\C$ of finite-dimensional $\Uqg$-modules, such a quantum torus already exists, denoted by $\Y_t$ here. Thus one natural condition on $\T_t$ is for it to contain $\Y_t$. We show it is the case in Proposition \ref{propinclusiontores}.

We start this section by recalling the definition and some properties of $\Y_t$. Here we use the same quantum torus as in \citep{AAqtC}, which is slightly different from the one used in \citep{QVtA} and \citep{PSqGR}.

		\subsection{The torus $\Y_t$}\label{sectdefYt}

In this section, we consider $\Uqg$-modules and no longer $\Uqb$-modules. We have seen in Section \ref{sectfd} that for finite-dimension representations, these settings were not too different.

As seen in (\ref{theoKC}), the Grothendieck ring of $\C$ can be seen as a subring of a ring of Laurent polynomials
\begin{equation*}
K_0(\C) \subseteq \hat{\mathcal{Y}} = \mathbb{Z}[Y_{i,a}]_{i\in I, a\in \mathbb{C}^\times}
\end{equation*}

In order to define a $t$-deformed non-commutative version of this Grothendieck ring, one first needs a non-commutative, $t$-deformed version of $\hat{\mathcal{Y}}$, denoted by $\Y_t$.

Following \citep{AAqtC}, we define
\begin{equation}
\mathcal{Y} := \mathbb{Z}[Y_{i,q^r}^\pm\mid (i,r)\in \hat{I}],
\end{equation}
the Laurent polynomial ring generated by the commuting variables $Y_{i,q^r}$. 

Let $(\mathcal{Y}_t, \ast)$ be the $\Z(t)$-algebra generated by the $(Y_{i,q^r}^\pm)_{(i,r)\in \hat{I}}$, with the $t$-commutations relations:
\begin{equation}\label{relY}
Y_{i,q^r}\ast Y_{j,q^s} = t^{\mathcal{N}_{i,j}(r-s)}Y_{j,q^s}\ast Y_{i,q^r},
\end{equation}
where $\mathcal{N}_{i,j} : \Z \to \Z$ is the antisymmetrical map, defined by
\begin{equation}\label{eqCN}
\mathcal{N}_{i,j}(m)=\tilde{C}_{i,j}(m+1) - \tilde{C}_{i,j}(m-1), \quad \forall m \geq 0,
\end{equation}
using the notations from Section \ref{sectQCM}. 
\begin{ex}\label{exsl2-2}
If we continue Example \ref{exsl2-1}, for $\mathfrak{g}=\mathfrak{sl}_2$, in this case, $\hat{I}=(1,2\Z)$, for $r\in\Z$, one has
\begin{equation}\label{relsl2Y}
Y_{1,2r}\ast Y_{1,2s} = t^{2(-1)^{s-r}}Y_{1,2s}\ast Y_{1,2r}, \quad \forall s>r>0.
\end{equation}
\end{ex}

The $\mathbb{Z}(t)$-algebra $\mathcal{Y}_t$ is viewed as a quantum torus of infinite rank.

Let us extend this quantum torus $\mathcal{Y}_t$ by adjoining a fixed square root $t^{1/2}$ of $t$:
\begin{equation}
\Z(t^{1/2})\otimes_{\Z(t)}\mathcal{Y}_t.
\end{equation}

By abuse of notation, the resulting algebra will still be denoted $\mathcal{Y}_t$.

For a family of integers with finitely many non-zero components $(u_{i,r})_{(i,r)\in \hat{I}}$, define the \emph{commutative monomial} $\prod_{(i,r)\in \hat{I}} Y_{i,q^r}^{u_{i,r}}$ as 
\begin{equation}
\prod_{(i,r)\in \hat{I}} Y_{i,q^r}^{u_{i,r}} :=t^{\frac{1}{2}\sum_{(i,r)<(j,s)}u_{i,r}u_{j,s}\mathcal{N}_{i,j}(r,s)} \overrightarrow{\bigast}_{(i,r)\in \hat{I}} Y_{i,q^r}^{u_{i,r}},
\end{equation}
where on the right-hand side an order on $\hat{I}$ is chosen so as to give meaning to the sum, and the product $\ast$ is ordered by it (notice that the result does not depend on the order chosen).

The commutative monomials form a basis of the $\Z(t^{1/2})$-vector space $\mathcal{Y}_t$. 

%The non-commutative product of two commutative monomials $m_1$ and $m_2$ in $\Yt$ is given by:
%\begin{equation}
%m_1 \ast m_2 = t^{D(m_1,m_2)}m_2 \ast m_1 = t^{\frac{1}{2}D(m_1,m_2)}m_1m_2,
%\end{equation}
%where $m_1m_2$ denotes the commutative product of the monomials, and
%\begin{equation}
%D(m_1,m_2)=\sum_{(i,r),(j,s)\in \hat{I}}u_{i,r}(m_1)u_{j,s}(m_2)\mathcal{N}_{i,j}(r,s).
%\end{equation}

	\subsection{The torus $\T_t$}\label{sectTt}

We now want to extend the quantum torus $\Yt$ to a larger non-commutative algebra $\T_t$ which would contain at least all the $\ell$-weights, and possibly all the candidates for the $\qt$-characters of the modules in the category $\cO^+_\Z$.

In particular, $\T_t$ contains the $\pmb\Psi_{i,q^r}$, for $(i,r)\in \hat{I}$, and these $t$-commutes with a relation compatible with the $t$-commutation relation between the $Y_{i,q^{r+1}}$ (\ref{relY}).

\begin{rem}
One notices that there is a shift of parity between the powers of $q$ in the $Y$'s and the $\pmb\Psi$'s. From now on, we will consider the $\pmb\Psi_{i,q^r}$ and the $Y_{i,q^{r+1}}$, for $(i,r)\in \hat{I}$.
\end{rem}

We start as in Section \ref{sectdefYt}. First of all, define
\begin{equation}
\mathcal{T} : =\Z\left[ z_{i,r}^\pm \mid (i,r)\in\hat{I} \right],
\end{equation}
the Laurent polynomial ring generated by the commuting variables $z_{i,r}$.

Then, build a $t$-deformation $T_t$ of $\mathcal{T}$, as the $\Z[t^{\pm 1}]$-algebra generated by the $z_{i,r}^\pm$, for $(i,r)\in\hat{I}$, with a non-commutative product $\ast$, and the $t$-commutations relations
\begin{equation}\label{tcomm}
z_{i,r}\ast z_{j,s} = t^{\F_{ij}(s-r)} z_{j,s}\ast z_{i,r}, \quad \left( (i,r),(j,s)\in\hat{I} \right),
\end{equation}
where, for all $i,j\in I$, $\mathcal{F}_{ij}: \Z \to \Z$ is a anti-symmetrical map such that, for all $m\geq 0$,
\begin{equation}
\mathcal{F}_{ij}(m) = - \sum_{\substack{k\geq 1 \\ m \geq 2k-1}}\tilde{C}_{ij}(m -2k+1).
\end{equation}
Now, let
\begin{equation}
\mathcal{T}_t:=\Z[t^{\pm 1/2}]\otimes_{\Z[t^{\pm 1}]}T_t.
\end{equation}

Similarly, we define the commutative monomials in $\T_t$ as, 
\begin{equation}
\prod_{(i,r)\in \hat{I}} z_{i,q^r}^{v_{i,r}} := t^{\frac{1}{2}\sum_{(i,r)<(j,s)}v_{i,r}v_{j,s}\mathcal{F}_{i,j}(r,s)} \overrightarrow{\bigast}_{(i,r)\in \hat{I}}z_{i,q^r}^{v_{i,r}}.
\end{equation}

This based quantum torus will be enough to define a structure of quantum cluster algebra, but for it to contain the quantum Grothendieck ring of  the category $\cO^+_\Z$, one needs to extend it. In order to do that, we draw inspiration from Section \ref{sectKO}. Recall the definition of $\chi$ from (\ref{chiZ}). We extend it to the $\E$-algebra morphism $\chi: \mathcal{T}_t\otimes_{\Z}\E \to \E$ defined by imposing $\chi(t^{\pm 1/2}) = 1$, as well as
\begin{equation*}
\chi(z_{i,r}^{\pm 1})  = \left[ \left( \frac{\mp r}{2}\right) \omega_i \right], \quad ((i,r) \in \hat{I}).
\end{equation*}

As before, for $z\in \mathcal{T}_t\otimes_{\Z}\E$, one writes $\chi(z)=\sum_{\omega\in P_\Q}z_\omega[\omega]$ and $\abs{\chi}(z)=\sum_{\omega\in P_\Q}\abs{z_\omega}[\omega]$. 

Define the completed tensor product
\begin{equation}\label{comptensorTt}
\T_t:=\mathcal{T}_t\hat{\otimes}_{\Z[t^{\pm 1/2}]}\E,
\end{equation}
of countable sums $\sum_{k\in\N}z_k$ of elements $z_k \in \mathcal{T}_t\otimes_{\Z}\E $, such that $\sum_{k\in\N}\abs{\chi}(z_k)$ is a countable sum of $\E$, as in Section \ref{sectKO}.

%Consider a subring of $\E_\ell$, defined in Section \ref{sectqchar}: let $\E_{\ell,\Z}$ be the ring of formal sums
%\begin{equation}
%c = \sum_{\pmb\Psi\in \Plr}c(\pmb\Psi)[\pmb\Psi] \in \E_\ell,
%\end{equation}
%such that, for all $\pmb\Psi \in\supp(c)$, the roots and poles of $\Psi_i$ are of the form $q^r$, for $(i,r)\in \hat{I}$.

%\begin{rem}
%All the $[\pmb\Psi]$ which appear in elements of $\E_{\ell,\Z}$ are of the form $[\omega]\cdot m$, with $\omega \in \E$ and $m$ a monomial in the $(\pmb\Psi_{i,q^r}^\pm)_{(i,r)\in \hat{I}}$. Thus, one has
%\begin{equation}
%\E_{\ell,\Z}= \E_\ell\otimes_{\Z[t^{\pm 1}]} \mathcal{T}.
%\end{equation}
%\end{rem}

%Let us endow $\E_{\ell,\Z}$ with a non-commutative product $\ast$, defined by, for all $(i,r),(j,s)\in \hat{I}$,
%\begin{equation}
%[\pmb\Psi_{i,q^r}]\ast [\pmb\Psi_{j,q^s}]= t^{\mathcal{F}_{i,j}(s-r)}[\pmb\Psi_{j,q^s}]\ast [\pmb\Psi_{i,q^r}],
%\end{equation}
%where $\mathcal{F}_{i,j}: \Z \to \Z$ is a antisymmetrical map such that, for all $m\geq 0$,
%\begin{equation}
%\mathcal{F}_{i,j}(m) = \left\lbrace \begin{array}{ll}
%						0  & \text{ if }  m<2 \\
%						- \sum_{k=0}^{\lfloor m/2 \rfloor-1}\tilde{C}_{i,j}(m -2k-1) & \text{otherwise}.
%\end{array}\right.
%\end{equation}

Consistently with the identification (\ref{identZL}), and the character of the $z_{i,r}^{\pm 1}$, we use the following notation, for $(i,r)\in\hat{I}$, 
\begin{equation}\label{Psizz}
[\pmb\Psi_{i,q^r}^{\pm 1}]:= z_{i,r}^{\pm 1}\left[\frac{\pm r}{2}\omega_i \right] \quad \in \T_t.
\end{equation}

\begin{prop}\label{propinclusiontores}
The identification
\begin{equation}\label{Yzz}
\J : Y_{i,q^{r+1}} \mapsto z_{i,r}z_{i,r+2}^{- 1} = [\omega_i][\pmb\Psi_{i,q^r}] [\pmb\Psi_{i,q^{r+2}}^{-1}],
\end{equation}
where the products on the right hand side are commutative, extends to a well-defined injective $\Z(t)$-algebra morphism $\J : \Y_t \to \T_t$.

%\begin{equation}
%\begin{array}{rcl}
%\Y_t & \to & \T_t \\
%Y_{i,q^{r+1}}^\pm & \mapsto & [\pm \omega_i] [\pmb\Psi_{i,q^r}]^{\pm 1}[\pmb\Psi_{i,q^{r+2}}]^{\mp 1},
%\end{array}
%\end{equation}
%defines an injective $\Z(t)$-algebra morphism.
\end{prop}

\begin{proof}
One needs to check that the images of the $Y_{i,q^{r+1}}$ satisfy (\ref{relY}). Thus, we need to show that, for all $(i,r),(j,s)\in \hat{I}$,
\begin{equation*}
\left( z_{i,r}z_{i,r+2}^{- 1} \right) \ast \left(  z_{j,s}z_{j,s+2}^{- 1} \right) = 
t^{\mathcal{N}_{i,j}(s-r)}\left( z_{j,s}z_{j,s+2}^{- 1} \right) \ast \left( z_{i,r}z_{i,r+2}^{- 1} \right),
\end{equation*}
which is equivalent to checking that:
\begin{equation}\label{eqFN}
2\mathcal{F}_{i,j}(s-r)-\mathcal{F}_{i,j}(s-r + 2) - \mathcal{F}_{i,j}(s-r-2) = \mathcal{N}_{i,j}(s-r).
\end{equation}
Suppose $s\geq r+2$, let $m=s-r$.
\begin{multline*}
2\mathcal{F}_{i,j}(m)-\mathcal{F}_{i,j}(m + 2) - \mathcal{F}_{i,j}(m-2) = - \sum_{\substack{k\geq 1 \\ m\geq 2k-1}}\tilde{C}_{ij}(m-2k+1) \\
+ \sum_{\substack{k\geq 0 \\ m\geq 2k-1}}\tilde{C}_{ij}(m-2k+1) + \sum_{\substack{k\geq 2 \\ m\geq 2k-1}}\tilde{C}_{ij}(m-2k+1) \\
= -\tilde{C}_{ij}(m-1) +\tilde{C}_{ij}(m+1) .
\end{multline*}
Thus $2\mathcal{F}_{i,j}(m)-\mathcal{F}_{i,j}(m + 2) - \mathcal{F}_{i,j}(m-2) =\mathcal{N}_{i,j}(m)$, using (\ref{eqCN}).

If $s=r+1$, the left-hand side of (\ref{eqFN}) is equal to
\begin{equation*}
3\F_{i,j}(1) - \F_{i,j}(3) = \tilde{C}_{ij}(2) = \mathcal{N}_{i,j}(1).
\end{equation*}
\end{proof}

%\begin{rem}
%For all $c=(c_i)_{i\in I} \in \E^I$, define $\T_t^c$ as the sub$\Ct$-algebra of $\T_t$ generated by the $z'_{i,r}:=c_i z_{i,r}$
%%%%%%%%%%%%%%%%%%%%%%%%%%%%%%%%%%%%%%%%%%%%%%%%%%%%%%%%%%%%%%%%%%%%%%%%%%%%%%%%%%%%%%%%%%%%%%%%%%%%%%%%%%%%%%%%%%%%%%%%%%%%

% the identification (\ref{Yzz})  
%\end{rem}

%Consider the Laurent polynomial ring
%\begin{equation}
%\Z\left[ [\pmb\Psi_{i,q^r}^\pm ] \mid (i,r)\in \hat{I} \right].
%\end{equation}
%For $(\pmb\Psi_k)_{k\in \mathbb{N}}$ a countable set of elements of this ring, 

%Thanks to Theorem \ref{theoctaO} and relation (\ref{eqPsiPsi'}), this is a subring of $\E_\ell$, by the map
%\begin{equation}
%\begin{array}{rcl}
%\T & \to & \E_\ell \\
%\pmb\Psi & \mapsto & \delta_{\pmb\Psi}.
%\end{array}
%\end{equation}

\begin{ex}\label{exsl2-3}
Let us continue Examples \ref{exsl2-1} and \ref{exsl2-2}. For all $r\in \Z$. One has
\begin{equation}\label{sl2tcom}
z_{1,2r}\ast z_{1,2s} = t^{f(s-r)}z_{1,2s}\ast z_{1,2r}~,\forall r,s\in\Z,
\end{equation}
where $f:\mathbb{Z}  \to \mathbb{Z}$ is antisymmetric and defined by
\begin{equation}\label{f}
f_{|\mathbb{N}} : m \mapsto \frac{(-1)^{m}-1}{2}.
\end{equation}
And this is compatible with the relations (\ref{relsl2Y}).
\end{ex}

\begin{defi}
Define the \emph{evaluation at }$t=1$ as the $\E$-morphism
\begin{equation}\label{eval}
\pi : \T_t \to \E_\ell,
\end{equation}
such that 
\begin{align*}
\pi(z_{i,r}) & = \left[ \frac{-r}{2}\omega_i\right][\pmb\Psi_{i,q^r}], \\
\pi(t^{\pm 1/2}) & = 1.
\end{align*}
\end{defi}

\begin{rem}
The identification (\ref{identZL}) is between the element $z_{i,r}\left[r\omega_i/2\right]$ and the class of the prefundamental representation $[L_{i,q^r}^+]$. But this identification is not compatible with the character $\chi$ defined in (\ref{chiZ}), as the character of $L_{i,q^r}^+$ is $\chi_i$, as in (\ref{chiqL}). Here, we choose to identify the variables $z_{i,r}$ with the highest $\ell$-weights of the prefundamental representations (up to a shift of weight), in particular, this identification is compatible with the character morphism $\chi$.
\end{rem}

	\section{Quantum Grothendieck rings}\label{sectQuantumGrothendieckRings}
	
 The aim of this section is to build $K_t(\cO^+_\mathbb{Z})$, a $t$-deformed version of the Grothendieck ring of the category $\cO^+_\Z$. This ring will be built inside the quantum torus $\T_t$, as a quantum cluster algebra.

Let us summarize the existing objects in this context in a diagram:

\begin{equation*}
\begin{tikzcd}[column sep=-4pt,row sep=20pt]
\C_\Z \arrow[d,dashrightarrow] & \subset & \cO^\pm_\Z \arrow[d,dashrightarrow]\\
K_0(\C_\Z) \arrow[d,rightsquigarrow] & \subset & K_0(\cO^\pm_\Z) \arrow[d,rightsquigarrow] & \simeq &  \mathcal{A}(\Gamma)\hat{\otimes}_\Z\E \arrow[dll,red]\\
K_t(\C_\Z) & \subset & \mbox{\large\pmb ? }
\end{tikzcd}
\end{equation*}
%\[
%\xymatrix{
%\C_\Z \ar@{}[r]|-*[@]{\subset} \ar@{.>}[d] & \cO^\pm_Z \ar@{.>}[d] & \\
%K_0(\C_\Z) \ar@{}[r]|-*[@]{\subset} \ar@{~>}[d] & K_0(\cO^\pm_Z) \ar@{~>}[d] \ar@{}[r]|-*[@]{\simeq} &\mathcal{A}(\Gamma)\hat{\otimes}_\Z\E \ar@{-->}[ld]\\
%K_t(\C_\Z) \ar@{}[r]|-*[@]{\subset} & \mbox{\large\pmb ? }&
%}
%\]

A natural idea to build a $t$-deformation of the Grothendieck ring  $\Kp$ is to use its cluster algebra structure and define a $t$-deformed quantum cluster algebra, as in Section \ref{sectqcluster}, with the same basis quiver. One has to make sure that the resulting object is indeed a subalgebra of the quantum torus $\T_t$.

		\subsection{The finite-dimensional case}

We start this section with some reminders regarding the quantum Grothendieck ring of the category of finite-dimensional $\Uqg$-modules.

This object was first discussed by Nakajima \citep{QVtA} and Varagnolo-Vasserot \citep{PSqGR} in the study of perverse sheaves. Then Hernandez gave a more algebraic definition, using $t$-analogs of screening operators \citep{tAOE},\citep{AAqtC}. This is the version we consider here, with the restriction to some specific tensor subcategory $\C_\Z$, as in \citep{QGR}.

			\subsubsection{Definition of the Quantum Grothendieck ring}\label{sectdefqGr}

%\subsubsection{The deformed Grothendieck ring $K_t(\C_\mathbb{Z})$}

As in Section \ref{sectOZpm}, consider $\C_\mathbb{Z}$ the full subcategory of $\C$ whose simple components have highest $\ell$-weights which are monomials in the $Y_{i,q^r}$, with $(i,r)\in \hat{I}$.

For $(i,r-1)\in \hat{I}$, define the commutative monomials
\begin{equation}
A_{i,r} := Y_{i,q^{r+1}}Y_{i,q^{r-1}}\prod_{j \sim i}Y_{j,q^r}^{-1} \quad \in \Yt.
\end{equation}

For all $i\in I$, let $K_{i,t}(\C_\Z)$ be the $\Z(t^{1/2})$-subalgebra of $\Yt$ generated by the 
\begin{equation}
Y_{i,q^r}(1 + A_{i,r+1}^{-1}), \quad Y_{j,q^s} \quad \left((i,r), (j,s) \in \hat{I}, j \neq i \right).
\end{equation}

Finally, as in \citep{AAqtC}, define

\begin{equation}
K_t(\C_\mathbb{Z}) := \bigcap_{i\in I}K_{i,t}(\C_\Z).
\end{equation}

\begin{rem}
Frenkel-Mukhin's algorithm \citep{CqC} allows for the computation of certain $q$-characters, in particular those of the fundamental representations. In \citep{AAqtC}, Hernandez introduced a $t$-deformed version of this algorithm to compute the $\qt$-characters of the fundamental representations, and thus to characterized the quantum Grothendieck ring as the subring of $\Y_t$ generated for those $\qt$-characters:
\begin{equation}
K_t(\C_\Z) = \left\langle [L(Y_{i,q^r})]_t \mid (i,r)\in \hat{I} \right\rangle.
\end{equation}
\end{rem}

		\subsubsection{$\qt$-characters in $K_t(\C_\Z)$}\label{qtdf}

Let us recall some more detailed results about the theory of $\qt$-characters for the modules in the category $\C_\Z$. 

Let $\mathcal{M}$ be the set of monomials in the variables $(Y_{i,q^{r+1}})_{(i,r)\in\hat{I}}$, also called \emph{dominant monomials}. From \citep{AAqtC} we know that for all dominant monomial $m$, there is a unique element $F_t(m)$ in $K_t(\C_\Z)$ such that $m$ occurs in $F_t(m)$ with multiplicity 1, and no other dominant monomial occurs in $F_t(m)$. These $F_t(m)$ form a $\mathbb{C}(t^{1/2})$-basis of $K_t(\C_\Z)$.

For all dominant monomial $m  = \prod_{(i,r)\in\hat{I}}Y_{i,q^{r+1}}^{u_{i,r}(m)} \in\mathcal{M}$, define
\begin{equation}
[M(m)]_t := t^{\alpha(m)} \overleftarrow{\bigast_{r\in\Z}} F_t\left( \prod_{i\in I} Y_{i,q^{r+1}}^{u_{i,r}(m)}\right) \quad \in K_t(\C_\Z),
\end{equation}
where $\alpha(m) \in \frac{1}{2}\Z$ is fixed such that $m$ appears with coefficient 1 in the expansion of $[M(m)]_t $ on the basis of the commutative monomials. The specialization at $t=1$ of $[M(m)]_t$ recovers the $q$-character $\chi_q(M(m))$ of the standard module $M(m)$.

Consider the bar-involution $\overline{\phantom{A}}$, the anti-automorphism of $\Y_t$ defined by:
\begin{equation}
\overline{t^{1/2}} = t^{-1/2}, \quad \overline{Y_{i,q^{r+1}}} = Y_{i,q^{r+1}}, \quad \left( (i,r) \in\hat{I} \right).
\end{equation}

\begin{rem}
The commutative monomials, defined in Section \ref{sectdefYt} are clearly bar invariant, as well as the subring $K_t(\C_\Z)$. 
\end{rem}

There is a unique family $\left\lbrace [L(m)]_t\in K_t(\C_\Z) \mid m \in \mathcal{M} \right\rbrace$ such that
\begin{enumerate}[(i)]
	\item \begin{equation}
	\overline{[L(m)]_t} = [L(m)]_t,
	\end{equation}
	\item \begin{equation}
	[L(m)]_t \in [M(m)]_t + \sum_{m'<m}t^{-1}\Z[t^{-1}][M(m')]_t,
	\end{equation}
where $m'\leq m$ means that $m(m')^{-1}$ is a product of $A_{i,r}$.
\end{enumerate} 

Lastly, we recall this result from Nakajima, proven using the geometry of quiver varieties.

\begin{theo}\citep{QVtA}
For all dominant monomial $m\in \mathcal{M}$, the specialization at $t=1$ of $[L(m)]_t$ is equal to $\chi_q(L(m))$.

Moreover, the coefficients of the expansion of $[L(m)]_t$ as a linear combination of products of $Y_{i,r}^{\pm 1}$ belong to $\N[t^{\pm 1}]$.
\end{theo}

Thus to all simple modules $L(\pmb \Psi)$ in $\C_\Z$ is associated an object $[L(m)]_t \in K_t(\C_\Z)$,  called the $\qt$\emph{-character}. It is compatible with the $q$-character of the representation.

\begin{rem}
With the cluster algebra approach, we shed a new light on this last positivity result. We interpret the $\qt$-characters of the fundamental modules (and actually all simple modules which are realized as cluster variables in $K_0(\cO^+_\Z)$) as quantum cluster variables (Conjecture \ref{conjchir}). Thus using Theorem \ref{theopos}, we recover the fact that the coefficients of their expansion on the commutative monomials in the $(Y_{i,r}^{\pm 1})$ belong to $\mathbb{N}[t^{\pm 1}]$.
\end{rem}
		
\begin{rem}
In order to fully extended this picture to the context of the category $\cO$, and implement a Kazhdan-Lusztig type algorithm to compute the $\qt$-characters of all simple modules, one would need an equivalent of the standard modules in this category. These do not exist in general. This question was tackled by the author in another paper \citep{LEA}, in which equivalent of standard modules where defined when $\g= \mathfrak{sl}_2$.
\end{rem}

	\subsection{Compatible pairs}\label{sectcompatiblepairs}

We now begin the construction of $\Ktp$.

First of all, to define a quantum cluster algebra, one needs a \emph{compatible pair}, as in Section \ref{sectcompair}. The basis quiver we consider here is the same quiver $\Gamma$ as before (see Section \ref{sectinftyquiver}).

Explicitly, the corresponding exchange matrix is the $\hat{I}\times \hat{I}$ skew-symmetric matrix $\tilde{B}$ such that, for all $\left((i,r),(j,s)\right) \in \hat{I}^2$,
\begin{equation}\label{coeffB}
\tilde{B}_{\left((i,r),(j,s)\right)} = \left\lbrace \begin{array}{lcl}
		1 & \text{ if } & i=j \text{ and } s=r+2 \\
			& & \text{ or } i\sim j \text{ and } s=r-1, \\
		-1 & \text{ if } & i=j \text{ and } s=r-2 \\
			& & \text{ or } i\sim j \text{ and } s=r+1,\\
		0 & \text{ otherwise}.		
\end{array}\right.
\end{equation}

%Now, let us define the other matrices $\Lambda$ and $(\Lambda_N)_{N\in\N^*}$, which are going to form compatible pairs together with $\tilde{B}$ and $\tilde{B}_N$.

Let $\Lambda$ be the $\hat{I}\times \hat{I}$ skew-symmetric infinite matrix encoding the $t$-commutation relations (\ref{tcomm}). Precisely, for $((i,r),(j,s))\in \hat{I}^2$ such that $s>r$,
\begin{equation}\label{Lambdaij}
\Lambda_{(i,r),(j,s)} = \mathcal{F}_{i,j}(s-r) = - \sum_{\substack{k\geq 1 \\ m \geq 2k-1}}\tilde{C}_{ij}(m -2k+1).
\end{equation}

\begin{rem}
In \citep{CABS}, it is noted that one can use sufficiently large finite subseed of $\Gamma$ instead of an infinite rank cluster algebra. For our purpose, the same statement stays true, but one has to check that the subquiver still forms a compatible pair with the torus structure. Hence, we have to give a more precise framework for the restriction to finite subseeds.
\end{rem}

For all $N\in \N^*$, define $\Gamma_N$, which is a finite slice of $\Gamma$ of length $2N+1$, containing an upper and lower row of frozen vertices. More precisely, define $\hat{I}_N$ and $\tilde{I}_N$ as 
\begin{align}
\hat{I}_N  & := \left\lbrace (i,r)\in \hat{I} \mid -2N +1 \leq r < 2N-1 \right\rbrace,\\
\tilde{I}_N  & := \left\lbrace (i,r)\in \hat{I} \mid -2N-1 \leq r < 2N+1 \right\rbrace \label{tildeIN}.
\end{align}
Then $\Gamma_N$ is the subquiver of $\Gamma$ with set of vertices $\tilde{I}_N$, where the vertices is $\tilde{I}_N\setminus \hat{I}_N$ are frozen (thus the vertices in $\hat{I}_N$ are the exchangeable vertices).

This way, all cluster variables of $\mathcal{A}(\Gamma)$ obtained from the initial seed after a finite sequence of mutations are cluster variables of the finite rank cluster algebra $\mathcal{A}(\Gamma_N)$, for $N$ large enough. With the same index restrict on $\tilde{B}$, we will be able to define a size increasing family of finite rank quantum cluster algebras.

\begin{ex}
Recall from Example \ref{exsl4} the infinite quiver $\Gamma$ when $\g=\mathfrak{sl}_{4}$. Then the quiver $\Gamma_N$ is the following:

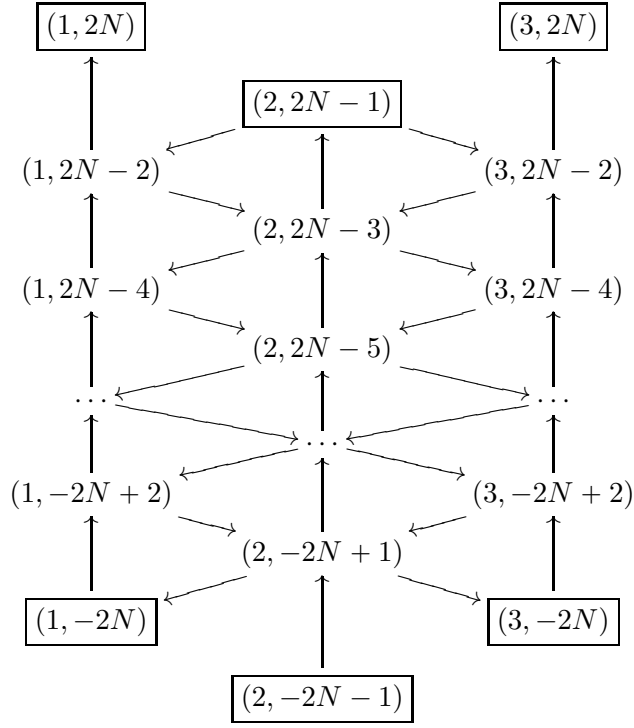
\[
\xymatrix@R=0.5em@C=1.5em{
\boxed{(1,2N)} &  & \boxed{(3,2N)} \\
& \boxed{(2,2N-1)} \ar[dl]\ar[dr] & \\
(1,2N-2)\ar[dr]\ar[uu] & & (3,2N-2) \ar[dl]\ar[uu] \\
& (2,2N-3) \ar[dl]\ar[dr]\ar[uu] & \\
(1,2N-4)\ar[dr]\ar[uu] & & (3,2N-4)\ar[dl]\ar[uu] \\
& (2,2N-5) \ar[dl]\ar[dr]\ar[uu] & \\
\cdots \ar[uu]\ar[dr] &   & \cdots\ar[uu]\ar[dl] \\
 &   \cdots \ar[uu]\ar[dr]\ar[dl]	& \\
 (1,-2N+2)\ar[dr]\ar[uu] & & (3,-2N+2)\ar[dl]\ar[uu] \\
& (2,-2N+1) \ar[dl]\ar[dr]\ar[uu] & \\
\boxed{(1,-2N)}\ar[uu] & & \boxed{(3,-2N)}\ar[uu] \\
& \boxed{(2,-2N-1)} \ar[uu] & 
}
\]
where the boxed vertices are frozen.
\end{ex}

For $N\in \N^*$, let $\tilde{B}_N$ be the corresponding exchange matrix. It is the $\tilde{I}_N\times \hat{I}_N$ submatrix of $\tilde{B}$, thus its coefficients are as in (\ref{coeffB}).

For all $N\in\N^*$, let $\Lambda_N$ be the $\tilde{I}_N\times \tilde{I}_N$ submatrix of $\Lambda$. It is a finite $(n(2N+1))^2$ skew-symmetric matrix, where $n$ is the rank of the simple Lie algebra $\mathfrak{g}$.

\begin{ex}\label{exD4}
For $\mathfrak{g}$ of type $D_4$, let us explicit a finite slice of $\Gamma$ of length 4, containing an upper and lower row of frozen vertices (which is thus not $\Gamma_1$, of length 3, nor $\Gamma_2$, of length 5):
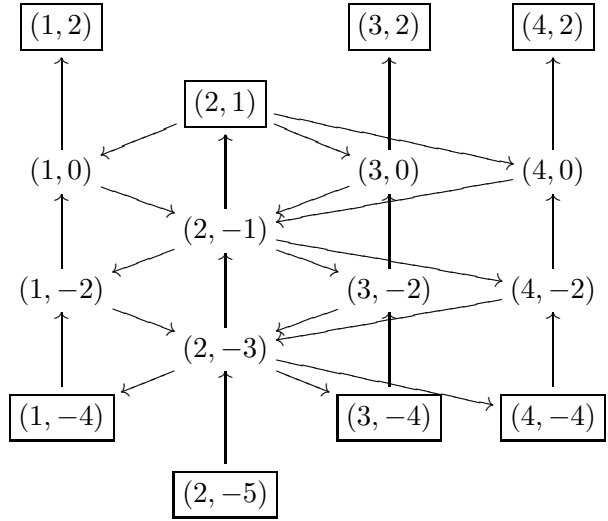
\[
\xymatrix@R=0.5em@C=1.5em{
\boxed{(1,2)} &  & \boxed{(3,2)} & \boxed{(4,2)} \\
& \boxed{(2,1)} \ar[dl]\ar[dr]\ar[drr] & & \\
(1,0)\ar[dr]\ar[uu] & & (3,0) \ar[dl]\ar[uu] & (4,0)\ar[dll]\ar[uu] \\
& (2,-1) \ar[dl]\ar[dr]\ar[uu]\ar[drr] & &  \\
(1,-2)\ar[dr]\ar[uu] & & (3,-2)\ar[dl]\ar[uu]  & (4,-2) \ar[dll]\ar[uu]\\
& (2,-3) \ar[dl]\ar[dr]\ar[uu]\ar[drr] & & \\
\boxed{(1,-4)}\ar[uu] & & \boxed{(3,-4)}\ar[uu] & \boxed{(4,-4)}\ar[uu]\\
& \boxed{(2,-5)} \ar[uu] & &
}
\]
If the set $\tilde{I} =\{ (i,r) \in \hat{I} \mid i \in \llbracket 1, 4 \rrbracket, -5\leq r\leq 2 \} $ is ordered lexicographically by $r$ then $i$ (reading order), the quiver is represented by the following exchange matrix:
\begin{equation}
\tilde{B} := \left( \begin{array}{rrrrrrrr}
-1 & 0 & 0 & 0 & 0 & 0 & 0 & 0 \\
0 & -1 & 0 & 0 & 0 & 0 & 0 & 0 \\
0 & 0 & -1 & 0 & 0 & 0 & 0 & 0 \\
1 & 1 & 1 & -1 & 0 & 0 & 0 & 0 \\
0 & 0 & 0 & 1 & -1 & 0 & 0 & 0 \\
0 & 0 & 0 & 1 & 0 & -1 & 0 & 0 \\
0 & 0 & 0 & 1 & 0 & 0 & -1 & 0 \\
-1 & -1 & -1 & 0 & 1 & 1 & 1 & -1 \\
1 & 0 & 0 & -1 & 0 & 0 & 0 & 1 \\
0 & 1 & 0 & -1 & 0 & 0 & 0 & 1 \\
0 & 0 & 1 & -1 & 0 & 0 & 0 & 1 \\
0 & 0 & 0 & 1 & -1 & -1 & -1 & 0 \\
0 & 0 & 0 & 0 & 1 & 0 & 0 & -1 \\
0 & 0 & 0 & 0 & 0 & 1 & 0 & -1 \\
0 & 0 & 0 & 0 & 0 & 0 & 1 & -1 \\
0 & 0 & 0 & 0 & 0 & 0 & 0 & 1 
\end{array}\right)
\end{equation}
The principal part $B$ of $\tilde{B}$ is the square submatrix  obtained by omitting the first 4 columns and the last 4 columns. One notices that $B$ is skew-symmetric. 

Moreover, using Formula (\ref{Lambdaij}), one can compute the corresponding  matrix $\Lambda$. We get the following $16\times 16$ skew-symmetric matrix (with the same order of $\tilde{I}$ as before):
\begin{equation}
\left(\begin{array}{rrrrrrrrrrrrrrrrrrrr}
0 & 0 & 0 & 0 & 1 & 0 & 0 & 1 & 1 & 1 & 1 & 2 & 2 & 1 & 1 & 2 \\
0 & 0 & 0 & 0 & 0 & 1 & 0 & 1 & 1 & 1 & 1 & 2 & 1 & 2 & 1 & 2 \\
0 & 0 & 0 & 0 & 0 & 0 & 1 & 1 & 1 & 1 & 1 & 2 & 1 & 1 & 2 & 2
\\
0 & 0 & 0 & 0 & 0 & 0 & 0 & 1 & 1 & 1 & 1 & 3 & 2 & 2 & 2 & 4 \\
-1 & 0 & 0 & 0 & 0 & 0 & 0 & 0 & 1 & 0 & 0 & 1 & 1 & 1 & 1 & 2
\\
0 & -1 & 0 & 0 & 0 & 0 & 0 & 0 & 0 & 1 & 0 & 1 & 1 & 1 & 1 & 2
\\
0 & 0 & -1 & 0 & 0 & 0 & 0 & 0 & 0 & 0 & 1 & 1 & 1 & 1 & 1 & 2 \\
-1 & -1 & -1 & -1 & 0 & 0 & 0 & 0 & 0 & 0 & 0 & 1 & 1 & 1 & 1 & 3 \\
-1 & -1 & -1 & -1 & -1 & 0 & 0 & 0 & 0 & 0 & 0 & 0 & 1 & 0 & 0 & 1 \\
-1 & -1 & -1 & -1 & 0 & -1 & 0 & 0 & 0 & 0 & 0 & 0 & 0 & 1 & 0 & 1 \\
-1 & -1 & -1 & -1 & 0 & 0 & -1 & 0 & 0 & 0 & 0 & 0 & 0 & 0 & 1 & 1 \\
-2 & -2 & -2 & -3 & -1 & -1 & -1 & -1 & 0 & 0 & 0 & 0 & 0 & 0 & 0 & 1 \\
-2 & -1 & -1 & -2 & -1 & -1 & -1 & -1 & -1 & 0 & 0 & 0 & 0 & 0 & 0 & 0 \\
-1 & -2 & -1 & -2 & -1 & -1 & -1 & -1 & 0 & -1 & 0 & 0 & 0 & 0 & 0 & 0 \\
-1 & -1 & -2 & -2 & -1 & -1 & -1 & -1 & 0 & 0 & -1 & 0 & 0 & 0 & 0 & 0 \\
-2 & -2 & -2 & -4 & -2 & -2 & -2 & -3 & -1 & -1 & -1 & -1 & 0 & 0 & 0 & 0 
\end{array}\right)
\end{equation}
From here, it is easy to check that the product $\tilde{B}^T\Lambda$ is of the desired form:
\begin{equation}\label{B2tL2}
\tilde{B}^T\Lambda = \left( \begin{array}{rrrrrrrrrrrrrrrr}
0 & 0 & 0 & 0 & -2 & 0 & 0 & 0 & 0 & 0 & 0 & 0 & 0 & 0 & 0 & 0 \\
0 & 0 & 0 & 0 & 0 & -2 & 0 & 0 & 0 & 0 & 0 & 0 & 0 & 0 & 0 & 0 \\
0 & 0 & 0 & 0 & 0 & 0 & -2 & 0 & 0 & 0 & 0 & 0 & 0 & 0 & 0 & 0 \\
0 & 0 & 0 & 0 & 0 & 0 & 0 & -2 & 0 & 0 & 0 & 0 & 0 & 0 & 0 & 0 \\
0 & 0 & 0 & 0 & 0 & 0 & 0 & 0 & -2 & 0 & 0 & 0 & 0 & 0 & 0 & 0 \\
0 & 0 & 0 & 0 & 0 & 0 & 0 & 0 & 0 & -2 & 0 & 0 & 0 & 0 & 0 & 0 \\
0 & 0 & 0 & 0 & 0 & 0 & 0 & 0 & 0 & 0 & -2 & 0 & 0 & 0 & 0 & 0 \\
0 & 0 & 0 & 0 & 0 & 0 & 0 & 0 & 0 & 0 & 0 & -2 & 0 & 0 & 0 & 0
\end{array}\right).
\end{equation}
Thus, $(\Lambda,\tilde{B})$ is a compatible pair.
\end{ex}

We show that this result is true in general. Furthermore, the specific form we obtain in Equation (\ref{B2tL2}) is what we get in general.

\begin{prop}\label{propComp}
The pairs $(\Lambda,\tilde{B})$ and $\left((\Lambda_N,\tilde{B}_N)\right)_{N\in\N^*}$ are compatible pairs, in the sense of structure condition for quantum cluster algebras.

Moreover, 
\begin{align*}
\tilde{B}^T\Lambda & = -2\id_{\hat{I}}, \\
\tilde{B}_N^T\Lambda_N & = -2\left( \begin{BMAT}(b,50pt,30pt){c|c|c}{c}
(0) & \id_{\hat{I}_N} & (0) 
\end{BMAT} \right).
\end{align*}
\end{prop}

\begin{proof}
Let $\left((i,r),(j,s)\right)\in\hat{I}^2$. Let us compute:
\begin{equation}
\left(\tilde{B}^T\Lambda\right)_{(i,r),(j,s)} = \sum_{(k,u)\in\hat{I}}b_{(k,u),(i,r)}\lambda_{(k,u),(j,s)}.
\end{equation}
This is a finite sum, as each vertex in $\Gamma$ is adjacent to a finite number of other vertices. 

Suppose first that $r\neq s$. Without loss of generality, we can assume that $r <s$. Then, using the definition of the matrix $\Lambda$ in (\ref{Lambdaij}) and the coefficients of $\tilde{B}$ in (\ref{coeffB}), we obtain
\begin{equation}\label{eqBtL}
\left(\tilde{B}^T\Lambda\right)_{(i,r),(j,s)} = - \tilde{C}_{ij}(s-r-1) - \tilde{C}_{ij}(s-r+1) + \sum_{k \sim i}\tilde{C}_{kj}(s-r).
\end{equation}
Now recall from Lemma \ref{lemCtildeC}, for all $(i,j)\in I^2$, 
\begin{equation*}
\tilde{C}_{ij}(m-1) + \tilde{C}_{ij}(m+1) - \sum_{k \sim i}\tilde{C}_{kj}(m)  = 0, \quad \forall m \geq 1.
\end{equation*}
Thus, for all$(i,j)\in I^2$ and $r<s$, equation (\ref{eqBtL}) gives:
\begin{equation}
\left(\tilde{B}^T\Lambda\right)_{(i,r),(j,s)} = 0.
\end{equation}
Suppose now that $r=s$. In that case,
\begin{equation*}
\left(\tilde{B}^T\Lambda\right)_{(i,r),(j,r)}  =  -2\tilde{C}_{ij}(1)= -2\delta_{i,j},
\end{equation*}
using the other result from Lemma \ref{lemCtildeC}. Thus,
\begin{equation}
\tilde{B}^T\Lambda  = -2\id_{\hat{I}}.
\end{equation}
Now, for all $N\in\N^*$, let $(i,r)\in \hat{I}_N$ and $(j,s)\in\tilde{I}_N$. Let us write:
\begin{equation}
\left(\tilde{B}^T\Lambda\right)_{(i,r),(j,s)} = \sum_{(k,u)\in\tilde{I}_N}b_{(k,u),(i,r)}\lambda_{(k,u),(j,s)}.
\end{equation}
As $(i,r)\in \hat{I}_N$ is not a frozen variable, the $(j,s)\in\hat{I}$  such that $b_{(k,u),(i,r)}\neq 0$ are all in $\tilde{I}_N$. Hence the rest of the reasoning is still valid, and the result follows.
\end{proof}

	\subsection{Definition of $\Ktp$}

Everything is now in place to define $\Ktp$. Recall the based quantum torus $\T_t$, defined in Section \ref{sectTt}. By construction, the associated skew-symmetric bilinear form $\Lambda$ identifies with the infinite skew-symmetric $\hat{I}\times\hat{I}$-matrix from the previous section:
\begin{equation}
\Lambda(\e_{(i,r)},\e_{(j,s)})=\Lambda_{(i,r),(j,s)} = \F_{ij}(s-r), \quad \left( (i,r),(j,s)\in\hat{I}\right)
\end{equation}
where $(\e_{(i,r)})_{(i,r)\in\hat{I}}$ is the standard basis of $\Z^{(\hat{I})}$.

Let $\F$ be the skew-field of fractions of $\T_t$. We define the toric frame $M: \Z^{(\hat{I})} \to \F\setminus\{0\}$ by setting
\begin{equation}
M(\e_{(i,r)}) = z_{i,r}\quad \in \F,\quad \forall (i,r)\in\hat{I}.
\end{equation} 

Then the infinite rank matrix $\Lambda_M$ satisfies
\begin{equation}
\Lambda_M = \Lambda.
\end{equation}

From the result of Proposition \ref{propComp}, 
\begin{equation}
\mathcal{S}= \left( M,\tilde{B} \right)
\end{equation}
is a quantum seed.

\begin{defi}
Let $\mathcal{A}_t(\Gamma)$ be the quantum cluster algebra associated to the mutation-equivalence class of the quantum seed $\mathcal{S}$.
\end{defi}

\begin{rem}
One could note that this is an infinite rank quantum cluster algebra, which it not covered by the definition given in Section \ref{sectclusterqcluster}. However, we have a sequence of quantum cluster algebras $(\mathcal{A}_t(\Gamma_N))_{N\in\N^*}$, built on the finite quivers $(\Gamma_N)_{N\in\N^*}$. As the mutation sequences are finite, one can always assume we are working in the quantum cluster algebra $\mathcal{A}_t(\Gamma_N)$, with $N$ large enough. 
\end{rem}

Fix $N\in\N^*$. Let $m=(2N+1)\times n$, where $n$ is the rank of the simple Lie algebra $\mathfrak{g}$. 

Consider $L_N$, the sub-lattice of $\T_t$ generated by the $z_{i,r}$, with $(i,r)\in\tilde{I}_N$ (recall the definition of $\tilde{I}_N$ in (\ref{tildeIN})). $L_N$ is of rank $m$. Consider the toric frame $M_N$ which is the restriction of $M$ to $L_N$. In that case,
\begin{equation*}
\Lambda_{M_N}=\Lambda_N, \quad \text{ from the previous section.}
\end{equation*}
Thus,
\begin{equation}
\mathcal{S}_N:= \left( M_N,\tilde{B}_N \right)
\end{equation}
is a quantum seed.

\begin{defi}
Let $\mathcal{A}_t(\Gamma_N)$ be the quantum cluster algebra associated to the mutation-equivalence class of the quantum seed $\mathcal{S}_N$.
\end{defi}

\begin{defi}
Define
\begin{equation}
\Ktp:= \mathcal{A}_t(\Gamma)\hat{\otimes}\E,
\end{equation}
where the tensor product is completed as in (\ref{comptensorTt}). The ring $\Ktp$ is a $\E[t^{\pm 1/2}]$-subalgebra of $\T_t$.

%it contains some particular objects we identify with the $\qt$-characters of the positive prefundamental representations, for all $(i,r)\in\hat{I}$:
%\begin{equation}
%z_{i,r}\otimes \left[\frac{r}{2}\omega_i\right] \approx [L_{i,q^r}^+]_t.
%\end{equation}
%Of course, the $\qt$-characters of the representations in the category $\cO^+_\Z$ are not defined in general yet.

For $N\in\N^*$, with the same completion of the tensor product, define 
\begin{equation}
K_t(\cO_{\Z,N}^+):=\mathcal{A}_t(\Gamma_N)\hat{\otimes}\E.
\end{equation}
\end{defi}

	\section{Properties of $\Ktp$} \label{sectPropertiesKtp}
			\subsection{The bar involution}
			
The bar involution defined on $\Y_t$ (see Section \ref{qtdf}) has a counterpart on the larger quantum torus $\T_t$. Besides, as $\qt$-character of simple modules in $\C_\Z$ are bar-invariant by definition, it is natural for $\qt$-characters of simple modules in $\cO^+_\Z$ to also be bar-invariant.

There is unique $\E$-algebra anti-automorphism of $\T_t$ such that
\begin{equation*}
\overline{t^{1/2}}=t^{-1/2},\quad \overline{z_{i,r}} = z_{i,r}, \quad \text{and } \overline{[\omega_i]}=[\omega_i], \quad ((i,r)\in \hat{I}).
\end{equation*}

What is crucial to note here is that this definition is compatible with the bar-involution defined in general on the quantum torus of any quantum cluster algebra (see \citep[Section 6]{qCA}). However, this bar-involution has an important property:  all cluster variables are invariant under the bar involution. 

\begin{prop}\label{propbarinv}
All elements of $\Ktp$ of the form $\chi_t\otimes 1$, where $\chi_t\in\mathcal{A}_t(\Gamma)$ is a cluster variable, are invariant under the bar-involution.
\end{prop}

		\subsection{Inclusion of quantum Grothendieck rings}\label{sectInclusion}
		
As stated earlier, one natural property we would want to be satisfied by the quantum Grothendieck ring $\Ktp$ is to include the already-existing quantum Grothendieck $K_t(\C_\Z)$ of the category $\C_\Z$.

Note that those rings are contained in quantum tori, which are included in one another by the injective morphism $\J$  from Proposition \ref{propinclusiontores}:

\begin{equation*}
\begin{tikzcd}[column sep=-4pt,row sep=20pt]
K_t(\C_\Z) \arrow[d,dashrightarrow] & \subset & \Y_t \arrow[d,hook,"\J"] \\
\Ktp & \subset & \T_t.
\end{tikzcd}
\end{equation*}

Thus it is natural to formulate the following Conjecture:
\begin{conj}\label{conj}
The injective morphism $\J$ restricts to an inclusion of the quantum Grothendieck rings
\begin{equation}
\J : K_t(\C_\Z)  \subset \Ktp.
\end{equation}
\end{conj}

Recall that the quantum Grothendieck ring $K_t(\C_\Z)$ is generated by the classes of the fundamental representations $[L(Y_{i,q^{r+1}})]_t$, for $(i,r) \in\hat{I}$ (see Section \ref{sectdefqGr}). Hence, in order to prove Conjecture \ref{conj}, it is enough to show that the images of these $[L(Y_{i,q^{r+1}})]_t$ belong to $\Ktp$.

In Example \ref{excmut} we saw how, when the $\g=\mathfrak{sl}_2$, the class of the fundamental representation $[L(Y_{1,q^{-1}})]$ could be obtained as a cluster variable in $\A(\Gamma)$ after one mutation in direction $(1,0)$. 

This fact is actually true in more generality, as seen in \citep{CABS}, in the proof of Proposition 6.1. Let us recall this process precisely.

Fix $(i,r)\in \hat{I}$. We first define a specific sequence of vertices in $\Gamma$, as in \citep{ACAA}. Recall the definition of the dual Coxeter number $h^{\vee}$.
\[
\begin{array}{|c|c|c|c|c|c|}
\hline
\g & A_n & D_n & E_6 & E_7 & E_8\\
\hline
h^\vee & n+1 & 2n-2 & 12 & 18 & 30\\
\hline
\end{array}
\]
Let $h'=\lceil h^\vee \rceil$. Fix an ordering $(j_1,\ldots,j_n)$ of the vertices of the Dynkin diagram of $\g$ by taking first $j_1=i$, then all vertices which appear with the same oddity as $i$ in $\hat{I}$ (the $j$ such that $(j,r)\in\hat{I}$), then the vertices which appear with a different oddity ($(j,r+1)\in\hat{I}$). For all $k \in \{2,\ldots,h'\}$, $j\in\{1,\ldots,n\}$, define the sequence $S_{j,k}$ of $k$ vertices of the column $j$ of $\Gamma$ in decreasing order: 
\begin{equation}
S_{j,k} = (j,r+2h'-\epsilon),(j,r+2h'-\epsilon-2),\ldots,(j,r+2h'-\epsilon-2k+2),
\end{equation}
where $\epsilon\in\{0,1\}$, depending of the oddity. Then define
\begin{equation}
S_k = \overrightarrow{\bigcup_{j}}S_{j,k},
\end{equation}
with the order defined before. Finally, let:
% For all $k \in \{1,\ldots,h'\}$, define the sequence of $nk$ vertices $S_k$:
%\begin{align*}
%S_k & = (j_1,r+2h'),(j_1,r+2h'-2),\ldots,(j_1,r+2h'-2k+1)\\
%	& (j_2,r+2h'),(j_2,r+2h'-2),\ldots,(j_2,r+2h'-2k+1) \\
%	& \cdots \\
%	& (j_r,r+2h'),(j_r,r+2h'-2),\ldots,(j_r,r+2h'-2k+1)\\
%	& (j_{r+1},r+2h'-1),(j_{r+1},r+2h'-3),\ldots,(j_{r+1},r+2h'-2k) \\
%	& \cdots \\
%	& (j_n,r+2h'-1),(j_n,r+2h'-3),\ldots,(j_n,r+2h'-2k).
%\end{align*}
%Notice that in $S_k$ we visit each column of $\Gamma$ once, with a sequence of $k$ vertices, starting at $(j,r+2h')$ (resp. $(j,r+2h'-1)$) and going down.
%Define 
\begin{equation*}
S = S_{h'} \cdots S_2  ~ (i,r+2h'),
\end{equation*}
by reading left to right and adding one last $(i,r+2h')$ at the end. 
\begin{ex}
For $\g$ of type $D_4$, and $(i,r)=(1,0)$, the sequence $S$ is 
\begin{align*}
S  = & (1,6) ~ (1,4) ~ (1,2) ~~ (3,6) ~ (3,4) ~ (3,2)\\
& (4,6) ~ (4,4) ~ (4,2) ~~ (2,5) ~ (2,3) ~ (2,1) \\ 
& (1,6) ~ (1,4) ~~ (3,6) ~ (3,4) ~~ (4,6) ~ (4,4)\\
& (2,5) ~ (2,3) ~~~ (1,6)
\end{align*}
\end{ex}

Using \citep[Theorem 3.1]{ACAA} and elements from the proof of Proposition 6.1 in \citep{CABS}, one gets the following result.
\begin{prop}\label{propchir}
Let $\chi_{i,r}$ be the cluster variable of $\A(\Gamma)$ obtained at the vertex $(i,r+2h')$ after following the sequence of mutations $S$, then, via the identification (\ref{identZL})
\begin{equation}
\chi_{i,r} \equiv [L(Y_{i,q^{r+1}})].
\end{equation}
To see this result differently, if one writes $\chi_{i,r}$ as a Laurent polynomial in the variables $(z_{j,s})$, then $\chi_{i,r}$ is in the image of $\J$, and 
\begin{equation}
\chi_{i,r}=\J(\chi_q(L(Y_{i,q^{r+1}})).
\end{equation}
\end{prop} 

\begin{ex}
Let $\g=\mathfrak{sl}_3$ and $(i,r)=(1,0)$. The sequence of vertices $S$ is
\begin{equation}
S= (1,4) ~ (1,2) ~~ (2,3) ~ (2,1) ~~ (1,4).
\end{equation}
Let us compute the sequence of mutations $S$:

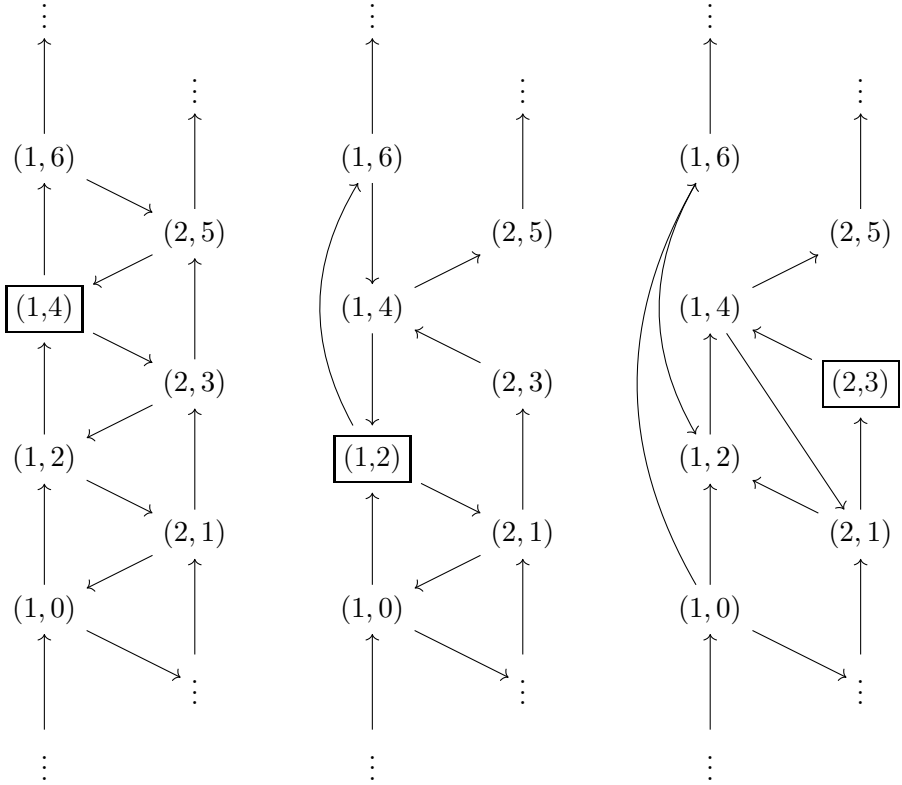
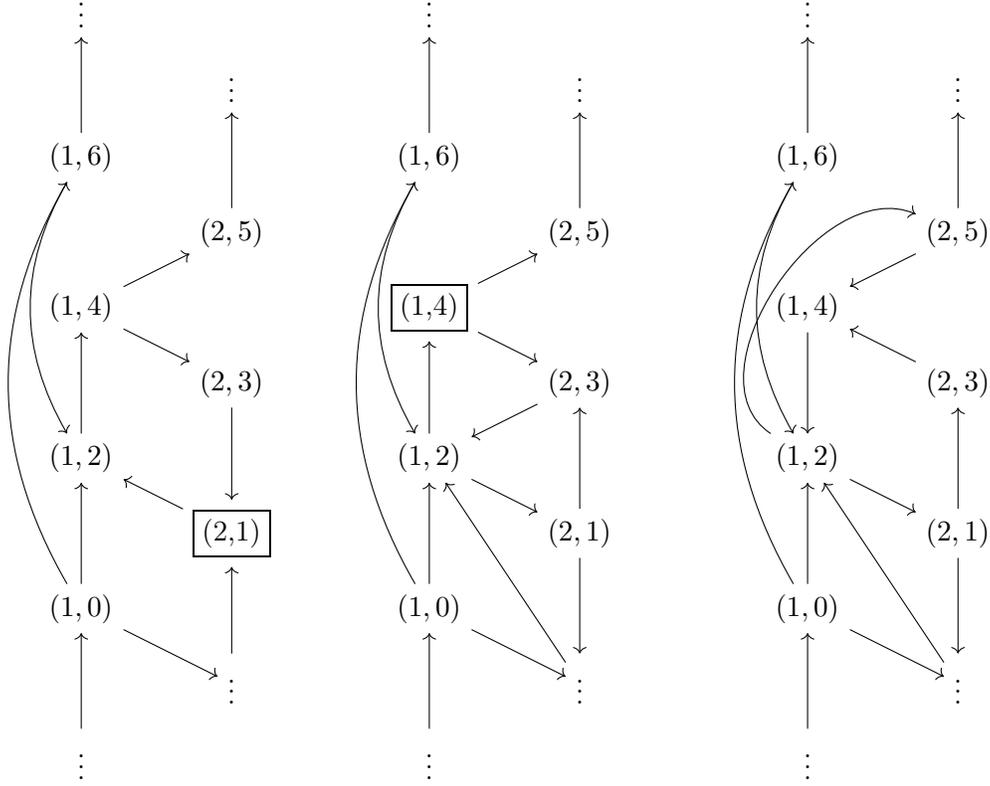
\begin{minipage}{4cm}
\begin{tikzpicture}
\node (18) at (0,2) {$\vdots$};
\node (16) at (0,0) {$(1,6)$};
\node (14) at (0,-2) {\boxed{$(1,4)$}};
\node (12) at (0,-4) {$(1,2)$};
\node (10) at (0,-6) {$(1,0)$};
\node (8) at (0,-8) {$\vdots$};

\node (17) at (2,1) {$\vdots$};
\node (15) at (2,-1) {$(2,5)$};
\node (13) at (2,-3) {$(2,3)$};
\node (11) at (2,-5) {$(2,1)$};
\node (9) at (2,-7) {$\vdots$};
%\node (7) at (2,-9) {$\vdots$};

\draw [->] (8) edge (10);
\draw [->] (10) edge (12);
\draw [->] (12) edge (14);
\draw [->] (14) edge (16);
\draw [->] (16) edge (18);

%\draw [->] (7) edge (9);
\draw [->] (9) edge (11);
\draw [->] (11) edge (13);
\draw [->] (13) edge (15);
\draw [->] (15) edge (17);

\draw [->] (16) edge (15);
\draw [->] (14) edge (13);
\draw [->] (12) edge (11);
\draw [->] (10) edge (9);

\draw [->] (15) edge (14);
\draw [->] (13) edge (12);
\draw [->] (11) edge (10);
\end{tikzpicture}
\end{minipage}
\begin{minipage}{4cm}
\begin{tikzpicture}
\node (18) at (0,2) {$\vdots$};
\node (16) at (0,0) {$(1,6)$};
\node (14) at (0,-2) {$(1,4)$};
\node (12) at (0,-4) {\boxed{$(1,2)$}};
\node (10) at (0,-6) {$(1,0)$};
\node (8) at (0,-8) {$\vdots$};

\node (17) at (2,1) {$\vdots$};
\node (15) at (2,-1) {$(2,5)$};
\node (13) at (2,-3) {$(2,3)$};
\node (11) at (2,-5) {$(2,1)$};
\node (9) at (2,-7) {$\vdots$};
%\node (7) at (2,-9) {$\vdots$};

\draw [->] (8) edge (10);
\draw [->] (10) edge (12);

\draw [->] (12) edge[bend left] (16);
\draw [->] (14) edge (12);
\draw [->] (16) edge (14);
\draw [->] (16) edge (18);

%\draw [->] (7) edge (9);
\draw [->] (9) edge (11);
\draw [->] (11) edge (13);
%\draw [->] (13) edge (15);
\draw [->] (15) edge (17);

%\draw [->] (16) edge (15);
\draw [->] (13) edge (14);
\draw [->] (12) edge (11);
\draw [->] (10) edge (9);

\draw [->] (14) edge (15);
%\draw [->] (13) edge (12);
\draw [->] (11) edge (10);
\end{tikzpicture}
\end{minipage}
\begin{minipage}{4cm}
\begin{tikzpicture}
\node (18) at (0,2) {$\vdots$};
\node (16) at (0,0) {$(1,6)$};
\node (14) at (0,-2) {$(1,4)$};
\node (12) at (0,-4) {$(1,2)$};
\node (10) at (0,-6) {$(1,0)$};
\node (8) at (0,-8) {$\vdots$};

\node (17) at (2,1) {$\vdots$};
\node (15) at (2,-1) {$(2,5)$};
\node (13) at (2,-3) {\boxed{$(2,3)$}};
\node (11) at (2,-5) {$(2,1)$};
\node (9) at (2,-7) {$\vdots$};
%\node (7) at (2,-9) {$\vdots$};

\draw [->] (8) edge (10);
\draw [->] (10) edge (12);

\draw [->] (10) edge[bend left] (16);
\draw [->] (16) edge[bend right] (12);
\draw [->] (12) edge (14);
%\draw [->] (16) edge (14);
\draw [->] (16) edge (18);

%\draw [->] (7) edge (9);
\draw [->] (9) edge (11);
\draw [->] (11) edge (13);
%\draw [->] (13) edge (15);
\draw [->] (15) edge (17);

%\draw [->] (16) edge (15);
\draw [->] (13) edge (14);
\draw [->] (11) edge (12);
\draw [->] (10) edge (9);

\draw [->] (14) edge (15);
%\draw [->] (13) edge (12);
%\draw [->] (11) edge (10);

\draw [->] (14) edge (11);
\end{tikzpicture}
\end{minipage}

\begin{minipage}{4.5cm}
\begin{tikzpicture}
\node (18) at (0,2) {$\vdots$};
\node (16) at (0,0) {$(1,6)$};
\node (14) at (0,-2) {$(1,4)$};
\node (12) at (0,-4) {$(1,2)$};
\node (10) at (0,-6) {$(1,0)$};
\node (8) at (0,-8) {$\vdots$};

\node (17) at (2,1) {$\vdots$};
\node (15) at (2,-1) {$(2,5)$};
\node (13) at (2,-3) {$(2,3)$};
\node (11) at (2,-5) {\boxed{$(2,1)$}};
\node (9) at (2,-7) {$\vdots$};
%\node (7) at (2,-9) {$\vdots$};

\draw [->] (8) edge (10);
\draw [->] (10) edge (12);

\draw [->] (10) edge[bend left] (16);
\draw [->] (16) edge[bend right] (12);
\draw [->] (12) edge (14);
%\draw [->] (16) edge (14);
\draw [->] (16) edge (18);

%\draw [->] (7) edge (9);
\draw [->] (9) edge (11);
\draw [->] (13) edge (11);
%\draw [->] (13) edge (15);
\draw [->] (15) edge (17);

%\draw [->] (16) edge (15);
\draw [->] (14) edge (13);
\draw [->] (11) edge (12);
\draw [->] (10) edge (9);

\draw [->] (14) edge (15);
%\draw [->] (13) edge (12);
%\draw [->] (11) edge (10);

%\draw [->] (14) edge (11);
\end{tikzpicture}
\end{minipage}
\begin{minipage}{4.5cm}
\begin{tikzpicture}
\node (18) at (0,2) {$\vdots$};
\node (16) at (0,0) {$(1,6)$};
\node (14) at (0,-2) {\boxed{$(1,4)$}};
\node (12) at (0,-4) {$(1,2)$};
\node (10) at (0,-6) {$(1,0)$};
\node (8) at (0,-8) {$\vdots$};

\node (17) at (2,1) {$\vdots$};
\node (15) at (2,-1) {$(2,5)$};
\node (13) at (2,-3) {$(2,3)$};
\node (11) at (2,-5) {$(2,1)$};
\node (9) at (2,-7) {$\vdots$};
%\node (7) at (2,-9) {$\vdots$};

\draw [->] (8) edge (10);
\draw [->] (10) edge (12);

\draw [->] (10) edge[bend left] (16);
\draw [->] (16) edge[bend right] (12);
\draw [->] (12) edge (14);
%\draw [->] (16) edge (14);
\draw [->] (16) edge (18);

%\draw [->] (7) edge (9);
\draw [->] (11) edge (9);
\draw [->] (11) edge (13);
%\draw [->] (13) edge (15);
\draw [->] (15) edge (17);

%\draw [->] (16) edge (15);
\draw [->] (14) edge (13);
\draw [->] (12) edge (11);
\draw [->] (10) edge (9);

\draw [->] (14) edge (15);
\draw [->] (13) edge (12);
%\draw [->] (11) edge (10);

%\draw [->] (14) edge (11);
\draw [->] (9) edge (12);
\end{tikzpicture}
\end{minipage}
\begin{minipage}{5cm}
\begin{tikzpicture}
\node (18) at (0,2) {$\vdots$};
\node (16) at (0,0) {$(1,6)$};
\node (14) at (0,-2) {$(1,4)$};
\node (12) at (0,-4) {$(1,2)$};
\node (10) at (0,-6) {$(1,0)$};
\node (8) at (0,-8) {$\vdots$};

\node (17) at (2,1) {$\vdots$};
\node (15) at (2,-1) {$(2,5)$};
\node (13) at (2,-3) {$(2,3)$};
\node (11) at (2,-5) {$(2,1)$};
\node (9) at (2,-7) {$\vdots$};
%\node (7) at (2,-9) {$\vdots$};

\draw [->] (8) edge (10);
\draw [->] (10) edge (12);

\draw [->] (10) edge[bend left] (16);
\draw [->] (16) edge[bend right] (12);
\draw [->] (14) edge (12);
%\draw [->] (16) edge (14);
\draw [->] (16) edge (18);

%\draw [->] (7) edge (9);
\draw [->] (11) edge (9);
\draw [->] (11) edge (13);
%\draw [->] (13) edge (15);
\draw [->] (15) edge (17);

%\draw [->] (16) edge (15);
\draw [->] (13) edge (14);
\draw [->] (12) edge (11);
\draw [->] (10) edge (9);

\draw [->] (15) edge (14);
%\draw [->] (13) edge (12);
%\draw [->] (11) edge (10);

%\draw [->] (14) edge (11);
\draw [->] (9) edge (12);
\draw [->] (12) edge[bend left, out=90, in=100]  (15);
\end{tikzpicture}
\end{minipage}

The associated cluster variables are:
\begin{align*}
z_{1,4}^{(1)}& = z_{1,2}z_{1,4}^{-1}z_{2,5} + z_{1,4}^{-1}z_{1,6}z_{2,3} , \\
z_{1,2}^{(1)} & = z_{1,0}z_{1,4}^{-1}z_{2,5} + z_{1,0}z_{1,2}^{-1}z_{1,4}^{-1}z_{1,6}z_{2,3} +  z_{1,2}^{-1} z_{1,6} z_{2,-1}, \\
z_{2,3}^{(1)} & = z_{2,1}z_{2,3}^{-1} + z_{1,2}z_{1,4}^{-1}z_{2,5}z_{2,3}^{-1} + z_{1,4}^{-1}z_{1,6} , \\
z_{1,4}^{(2)} & = z_{1,0}z_{1,2}^{-1} + z_{1,2}^{-1}z_{1,4}z_{2,1}z_{2,3}^{-1} + z_{2,3}^{-1}z_{2,5}.
\end{align*}
Thus, $\chi_{1,0}=z_{1,4}^{(2)}$ is in the image of $\J$, and 
\begin{equation}
\chi_{1,0}= \J(Y_{1,q} + Y_{1,q^3}^{-1}Y_{2,q^2} + Y_{2,q^4}^{-1}) = \J(\chi_q(L(Y_{1,q}))).
\end{equation}
Notice also that $z_{2,3}^{(1)}$ was already in the image of $\J$ and that $z_{2,3}^{(1)} = \J(\chi_q(L(Y_{2,q^2})))$.
\end{ex}
%\begin{proof}
%As in \citep{CABS}, consider the subquiver $G^-$ of $\Gamma$ with vertex set $\hat{I}^-=\{ (j,s) \mid s\leq r+2h'\}$, let us denote by $u_{j,s}$ the cluster variables of the initial seed of the cluster algebra $\A(G^-)$. 
%\end{proof}

%In more generality, in the proof of Proposition 6.1 in \citep{CABS} it is shown that the $q$-characters of all fundamental representations (expressed in the variables $z_{i,r}$) are cluster variable in the cluster algebra $\A(\Gamma)$ (the $q$-character morphism being injective, we identify the $q$-characters to the classes of the corresponding modules in the Grothendieck ring).

Thus, for each $(i,r) \in\hat{I}$, consider the quantum cluster variables $\tilde{\chi}_{i,r}\in \Ktp$ obtained from the initial quantum seed $(\pmb z, \Lambda)$ via the sequence of mutations $S$.

\begin{ex}\label{exmutsl2}
Suppose $\mathfrak{g}=\mathfrak{sl}_2$. Consider the quiver $\Gamma_1$ a well as the skew-symmetric matrix $\Lambda_1$,
\[
\Gamma_1 = \vcenter{\vbox{\xymatrix@R=0.5cm{
\boxed{(1,2)} \\
(1,0) \ar[u]\\
\boxed{(1,-2)} \ar[u]
}}}, \quad \Lambda_1 = \left( \begin{array}{ccc}
				0 & -1 & 0 \\
				1 & 0 & -1 \\
				0 & 1 & 0
\end{array}\right).
\]
As seen in Example \ref{excmut} (with a shift of quantum parameters), the fundamental representation $[L(Y_{1,q^{-1}})]$ is obtained in $K_0(\cO^+_Z)$ after one mutation at $(1,0)$ (here $S=(1,0)$).

%In this case, the $q$-character of the fundamental representation $L(Y_{1,q^{-1}})$ is 
%\begin{equation*}
%\chi_q(L(Y_{1,q^{-1}})) = Y_{1,q^{-1}} + Y_{1,q}^{-1}.
%\end{equation*}
The quantum cluster variable obtained after a quantum mutation at $(1,0)$, written with commutative monomials, is
\begin{align*}
\tilde{\chi}_{1,-2} & = z_{1,-2}z_{1,0}^{-1} +z_{1,2}z_{1,0}^{-1}  =  \J(Y_{1,q^{-1}}+Y_{1,q}^{-1}) = \J([L(Y_{1,q^{-1}})]_t),\\
 &=\J\left( Y_{1,q^{-1}}(1 + A_{1,1}^{-1}) \right) \quad \in \J(K_t(\C_\Z)),
\end{align*}
Thus, we note that in this particular case, the quantum cluster variable $\tilde{\chi}_{1,-2}$ recovers the $\qt$-character $[L(Y_{1,q^{-1}})]_t$ of the fundamental representation $L(Y_{1,q^{-1}})$.

In particular, Conjecture \ref{conj} is satisfied in this case.
\end{ex}

This example incites us to formulate another conjecture.
\begin{conj}\label{conjchir}
For all $(i,r)\in\hat{I}$, the quantum cluster variable $\tilde{\chi}_{i,r}$ recovers, via the morphism $\J$, the $\qt$-character of the fundamental representation $L(Y_{i,q^{r+1}})$:
\begin{equation}
\tilde{\chi}_{i,r} = \J\left( [L(Y_{i,q^{r+1}})]_t\right).
\end{equation}
\end{conj}

\begin{rem}
Notice that Conjecture \ref{conjchir} implies Conjecture \ref{conj}, and that Conjecture \ref{conjchir} is also satisfied when $\mathfrak{g}=\mathfrak{sl}_2$, from Example \ref{exmutsl2}.
\end{rem}

What can be said, in general, of the quantum cluster variables $\tilde{\chi}_{i,r}$ ?

\begin{prop}\label{propchi}
For all $(i,r)\in\hat{I}$, the quantum cluster variable $\tilde{\chi}_{i,r}$ satisfies the following properties:
\begin{enumerate}[(i)]
	\item invariant under the bar involution:
\begin{equation}
\overline{\tilde{\chi}_{i,r}} =\tilde{\chi}_{i,r}.
\end{equation}
	\item the coefficients of its expansion as a Laurent polynomial in the initial quantum cluster variables $\{z_{i,r}\}$ are Laurent polynomials in $t^{1/2}$ with non-negative integers coefficients:
\begin{equation}
\tilde{\chi}_{i,r} \in \bigoplus_{\pmb u =u_{i,r}\in\Z^{(\hat{I})}}\N[t^{\pm 1/2}]\pmb z^{\pmb u} .
\end{equation}
with $\pmb z^{\pmb u}= \prod_{(i,r)\in\hat{I}}z_{i,r}^{u_{i,r}}$ denoting the commutative monomial.
	\item its evaluation at $t=1$ (as seen in (\ref{eval})), recovers the $q$-character of the fundamental representation $L(Y_{i,q^{r+1}})$:
\begin{equation}
\pi(\tilde{\chi}_{i,r}) = \chi_q(L(Y_{i,q^{r+1}})).
\end{equation}
\end{enumerate}
\end{prop}

\begin{proof}
The first property is a direct consequence of Proposition \ref{propbarinv} and the second is a direct consequence of the positivity result of Theorem \ref{theopos}.

For the third property, notice we have used two evaluation maps so far, with the same notation.
\begin{itemize}
	\item[•] The evaluation map defined in (\ref{eval1}) on the bases quantum torus of a quantum cluster algebra:
	\begin{equation*}
\pi :  \mathcal{A}_t(M,\tilde{B}) \to  \Z[\tilde{\textbf{X}}^{\pm 1}],
	\end{equation*}
	\item[•] The evaluation map defined in (\ref{eval}) on $\T_t$:
	\begin{equation*}
\pi :  \T_t  \to  \E_\ell.
	\end{equation*}
\end{itemize}
These notations are coherent because the map $\pi$ from (\ref{eval}) is the evaluation map defined on a based quantum torus (of infinite rank) of a quantum cluster algebra, extended to a $\E$-morphism on $\T_t$. In this case, the Laurent polynomial ring $\Z[\tilde{\textbf{X}}^{\pm 1}]$ is $\Z[z_{i,r}^{\pm 1}\mid (i,r)\in\hat{I}]$, which becomes $\E[\pmb\Psi_{i,r}^{\pm 1}]$ after extension to a $\E$-morphism and via the identification (\ref{identZL}).

  Thus we can apply Corollary \ref{coreval} to this map $\pi$. As $\tilde{\chi}_{i,r}$ is a quantum cluster variable, its evaluation by $\pi$ is the cluster variable $\chi_{i,r}$, which is obtained from the initial seed $\textbf{z}$, via the same sequence of mutations $S$ (the initial seed and quantum seeds are fixed and identified by the evaluation $\pi$ on the quantum torus $\T_t$). By Proposition \ref{propchir},
\begin{equation}
\pi(\tilde{\chi}_{i,r}) = \chi_{i,r} = \chi_q(L(Y_{i,q^{r+1}})).
\end{equation} 
\end{proof}
		
These two properties imply that the $\tilde{\chi}_{i,r}$ are good candidates for the $\qt$-characters of the fundamental representations, as stated in Conjecture \ref{conj}.
		
		\subsection{$\qt$-characters for positive prefundamental representations}
	
Recall the $q$-characters of the positive prefundamental representations in (\ref{chiqL}), for all $i\in I, a\in\mathbb{C}^\times$,
\begin{equation*}
\chi_q(L_{i,a}^+)= [\pmb\Psi_{i,a}]\chi_i,
\end{equation*}
where $\chi_i\in \E$ is the (classical) character of $L_{i,a}^+$.

\begin{defi}
For $(i,r)\in\hat{I}$, define
\begin{equation}\label{defqtL}
[L_{i,q^r}^+]_t := [\pmb\Psi_{i,q^r}]\otimes \chi_i \quad \in \Ktp,
\end{equation}
using the notation from (\ref{Psizz}).
\end{defi}

\begin{rem} It is the quantum cluster variable obtained from the initial quantum seed, via the same sequence of mutations used to obtain $[L_{i,q^r}^+]$ in $K_0(\cO^+_\Z)$, which in this case, is no mutation at all.
\end{rem}

In particular, the evaluation of $[L_{i,q^r}^+]_t$ recovers the $q$-character of $[L_{i,q^r}^+]$:
\begin{equation}
\pi([L_{i,q^r}^+]_t) = [\pmb\Psi_{i,q^r}]\otimes \chi_i = \chi_q(L_{i,a}^+) \quad \in \E_\ell.
\end{equation}

	\section{Results in type A}\label{sectA}

Suppose in this section that the underlying simple Lie algebra $\g$ is of type $A$. 

	\subsection{Proof of the conjectures}
In this case, the situation of Example \ref{exmutsl2} generalizes.

\begin{theo}\label{theoA}
Conjecture \ref{conjchir} is satisfied in this case.
\end{theo}

In this case, the key ingredient is the following well-known result (see for exemple \citep[Section 11]{QAAVW}, and references therein).
\begin{theo}\label{theothinA}
When $\g$ is of type $A$, all $\ell$-weight spaces of all fundamental representations $L(Y_{i,a})$ are of dimension 1.
\end{theo}

\begin{proof}
Fix $(i,r)\in\hat{I}$. From the second property of Proposition \ref{propchi}, we know that $\tilde{\chi}_{i,r}$ can be written as
\begin{equation}\label{eqchi}
\tilde{\chi}_{i,r}= \sum_{\pmb u\in\Z^{(\hat{I})}} P_{\pmb u}(t^{1/2}) \pmb z^{\pmb u},
\end{equation}
where the $P_{\pmb u}(t^{1/2})$ are Laurent polynomials with non-negative integer coefficients. Using the third property of Proposition \ref{propchi}, we deduce the evaluation at $t=1$ of equality (\ref{eqchi}):
\begin{equation}
\chi_q(L(Y_{i,q^{r-1}})) = \sum_{\pmb u\in\Z^{(\hat{I})}} P_{\pmb u}(1) \prod_{(i,r)\in\hat{I}}\left([\pmb\Psi_{i,q^r}][-r\omega_i/2]\right)^{u_{i,r}} \quad \in \E_\ell.
\end{equation}
From the above theorem, this decomposition is multiplicity-free. Thus, the non-zero coefficients $P_{\pmb u}(t^{1/2})$ are of the form $t^{k/2}$, with $k\in\Z$. Finally, as $\chi_{i,r}$ is bar-invariant, from the first property of Proposition \ref{propchi}, and the $\pmb z^{\pmb u}$ are also bar-invariant as commutative monomials, we know that the Laurent polynomials $P_{\pmb u}(t^{1/2})$ are even functions:
\begin{equation}
P_{\pmb u}(-t^{1/2})=P_{\pmb u}(t^{1/2}).
\end{equation}
Thus the variable $t^{1/2}$ does not explicitly appear in the decomposition (\ref{eqchi}), and so:
\begin{align*}
\chi_{i,r} & = \sum_{\pmb u\in\Z^{(\hat{I})}} P_{\pmb u}(1) \pmb z^{\pmb u},\\
	& = \J \left(\chi_q(L(Y_{i,q^{r-1}}))\right).
\end{align*}
Moreover, with the same arguments, as $[L(Y_{i,q^{r-1}})]_t$ is bar-invariant by definition,
\begin{equation}
[L(Y_{i,q^{r-1}})]_t = \chi_q(L(Y_{i,q^{r-1}})),
\end{equation}
written in the basis of the commutative monomials.

Hence we recover the fact that the quantum cluster variable $\chi_{i,r}$ is equal, via the inclusion map $\J$, to the $\qt$-character of $L(Y_{i,q^{r-1}})$ and Conjecture \ref{conjchir} is satisfied.
\end{proof}

		\subsection{An remarkable subalgebra in type $A_1$}

When $\g=\mathfrak{sl}_2$, we can make explicit computations. Recall the formula (\ref{sl2tcom}) of the quantum torus from Example \ref{exsl2-3}:

\begin{equation*}
z_{1,2r}\ast z_{1,2s} = t^{f(s-r)}z_{1,2s}\ast z_{1,2r}~,\forall r,s\in\Z,
\end{equation*}
where $f:\mathbb{Z}  \to \mathbb{Z}$ is antisymmetric and defined by
\begin{equation*}
f_{|\mathbb{N}} : m \mapsto \frac{(-1)^{m}-1}{2}.
\end{equation*}

For all $r\in\Z$, the $\qt$-character of the prefundamental representation $L_{1,q^{2r}}^+$ defined in (\ref{defqtL}) is
\begin{equation*}
[L_{1,q^{2r}}^+]_t = [\pmb\Psi_{1q^{2r}}]\chi_1 .
\end{equation*}

\begin{prop}
With these $\qt$-characters, we can write a $t$-deformed version of the Baxter relation (\ref{TQ'}), for all $r\in\Z$,
\begin{equation}\label{baxterq}
[L(Y_{1,q^{2r-1}})]_t\ast[L_{1,q^{2r}}^+]_t = t^{-1/2}[\omega_1][L_{1,q^{2r-2}}^+]_t + t^{1/2}[-\omega_1][L_{1,q^{2r+2}}^+]_t. 
\end{equation}
\end{prop}
We call this relation the \emph{quantized Baxter relation}.

\begin{rem}
If we identify the variables $Y_{1,q^{2r}}$ and their images through the injection $\J$, this relation is actually the exchange relation related to the quantum mutation in Example \ref{exmutsl2} (for a generic quantum parameter $q^{2r}$).
\end{rem}

Now consider the quantum cluster algebra $\A(\Lambda_1,\Gamma_1)$, with notations from Section \ref{sectcompatiblepairs} ($\Lambda_1$ and $\Gamma_1$ are given explicitly in Example \ref{exmutsl2}).

It is a quantum cluster algebra of finite type (if we remove the frozen vertices from the quiver, we get just one vertex, which is a quiver of type $A_1$). It has two quantum clusters, containing the two frozen variables $z_{1,2}, z_{1,-2}$ and the mutable variables $z_{1,0}$ and $z_{1,0}^{(1)}$, respectively. Thus, it is generated as a $\Ctt$-algebra by
\begin{equation}
\begin{array}{ll}
E := [L(Y_{1,q^{-1}})]_t \enskip (=z_{1,0}^{(1)}) , & F:= [L_{1,1}^+]_t \enskip (=z_{1,0}), \\
K:= [\omega_1][L_{1,q^{-2}}^+]_t \enskip (=z_{1,-2}), & K':= [-\omega_1][L_{1,q^{2}}^+]_t \enskip (=z_{1,2}).
\end{array}
\end{equation}

This algebra is a quotient of a well-known $\Ctt$-algebra.

Let $q$ be a formal parameter. The quantum group $\mathcal{U}_q(\mathfrak{sl}_2)$ can be seen as the quotient
\begin{equation}\label{UqD2}
\mathcal{U}_q(\mathfrak{sl}_2) = \mathfrak{D}_2/\left\langle KK'=1\right\rangle,
\end{equation}
where $\mathfrak{D}_2$ is the $\mathbb{C}(q)$-algebra with generators $ E,F,K,K'$ and relations:
\begin{equation}\label{relD2}
\begin{array}{ll}
KE=q^2EK, & K'E=q^{-2}EK' \\
KF=q^{-2}FK, & K'F=q^2FK'\\
KK' = K'K, & \text{ and } [E,F] = (q-q^{-1})(K-K').
\end{array}
\end{equation}

\begin{rem}\begin{itemize}
	\item[•] As in \citep[Remark 3.1]{SS}, notice that the last relation in (\ref{relD2}) is not the usual relation
\begin{equation*}
[e,f] = \frac{K-K'}{q-q^{-1}}.
\end{equation*}
But both presentations are equivalent, given the change of variables 
\begin{equation*}
E=(q-q^{-1})e, \quad F=(q-q^{-1})f.
\end{equation*}

	\item[•] $\mathfrak{D}_2$ is the \emph{Drinfeld double} \citep{QGD} of the Borel subalgebra of $\mathcal{U}_q(\mathfrak{sl}_2)$ (the subalgebra generated by $K,E$).
\end{itemize}
\end{rem}

\begin{prop}\label{propquotient}
The $\Ctt$-algebra $\A(\Lambda_1,\Gamma_1)$ is isomorphic to the quotient of the Drinfeld double $\mathfrak{D}_2$ of parameter $-t^{1/2}$, 
\begin{equation}
\A(\Lambda_1,\Gamma_1) \xrightarrow{\sim} \mathfrak{D}_2/C_{-t^{1/2}},
\end{equation}
where $C_{-t^{1/2}}$ is the quantized Casimir element: 
\begin{equation}\label{casimir}
C_{-t^{1/2}} := E F -t^{1/2}K-t^{1/2}K'.
\end{equation}
\end{prop}

\begin{proof}
One has, in $\A(\Lambda_1,\Gamma_1)$, 
\begin{equation}\label{relEF}
E\ast F = t^{-1/2}K + t^{1/2}K'.
\end{equation}
This is the quantized Baxter relation (\ref{baxterq}). Thus,
\begin{equation*}
[E,F] = (-t^{1/2} +t^{-1/2})(K-K').
\end{equation*}
We check that the other relations in (\ref{relD2}) are also satisfied using the structure of the quantum torus $\T_t$ (which is given explicitly in Example \ref{exsl2-3}). 

Hence the map 
\begin{equation*}
\A(\Lambda_1,\Gamma_1) \xrightarrow{\theta} \mathfrak{D}_2,
\end{equation*}
sending generators to generators is well-defined and descends onto the quotient
\begin{equation*}
\A(\Lambda_1,\Gamma_1) \to \mathfrak{D}_2/C_{-t^{1/2}}.
\end{equation*}
Moreover, from \citep{LICM}, the cluster monomials in a given cluster in a cluster algebra are linearly independent. In this case, the quantum cluster algebra $\A(\Lambda_1,\Gamma_1)$ is of type $A_1$ (without frozen variables), thus of finite-type. It has two (quantum) clusters : $(E,K,K')$ and $(F,K,K')$. Thus, the set of commutative quantum cluster monomials
\begin{equation}\label{EKKKKF}
\left\lbrace E^\alpha K^\beta K'^\gamma \mid \alpha,\beta,\gamma \in \Z \right\rbrace \cup \left\lbrace K^\beta K'^\gamma F^\alpha \mid \alpha,\beta,\gamma \in \Z \right\rbrace, 
\end{equation}
forms a $\Ctt$-basis of $\A(\Lambda_1,\Gamma_1)$.

Consider the PBW basis of $\mathfrak{D}_2$:
\begin{equation}
\left\lbrace E^\alpha K^\beta K'^\gamma F^\delta \mid \alpha,\beta,\gamma,\delta \in \Z \right\rbrace.
\end{equation}
From the expression of the Casimir element $C_{-t^{1/2}}$ (\ref{casimir}), we deduce a $\Ctt$-basis of $\mathfrak{D}_2/C_{-t^{1/2}}$, of the same form as (\ref{EKKKKF}):
\begin{equation}
\left\lbrace E^\alpha K^\beta K'^\gamma \mid \alpha,\beta,\gamma \in \Z \right\rbrace \cup \left\lbrace K^\beta K'^\gamma F^\alpha \mid \alpha,\beta,\gamma \in \Z \right\rbrace, .
\end{equation}
Hence, the map $\theta$ sends a basis to a basis, thus it is isomorphic.

\end{proof}

This result should be compared with the recent work of Schrader and Shapiro \citep{SS}, in which they recognize the same structure of $\mathfrak{D}_2$ in an algebra built on a quiver, with some quantum $\mathcal{X}$-cluster algebra structure. In their work, they generalized this result in type $A$ (Theorem 4.4). Ultimately, they obtain an embedding of the whole quantum group $\mathcal{U}_q(\mathfrak{sl}_n)$ into a quantum cluster algebra. The result of Proposition \ref{propquotient}, together with their results, gives hope that one could find a realization of the quantum group $\mathcal{U}_q(\g)$ as a quantum cluster algebra, related to the representation theory of $\Uqg$.

Furthermore, define in this case $\cO_1^+$, the subcategory of $\cO^+_\Z$ of objects whose image in the Grothendieck ring $K_0(\cO^+_\Z)$ belongs to the subring generated by $[L_{1,q^{-2}}^+],[L_{1,1}^+],[L_{1,q^{2}}^+]$ and $[L(Y_{1,q^{-1}})]$. Then $\cO_1^+$ is a monoidal category. 

From the classification of simple modules when $\g=\mathfrak{sl}_2$ in \citep[Section 7]{CABS}, we know that the only prime simple modules in $\cO_1^+$ are
\begin{equation}
L_{1,q^{-2}}^+,L_{1,1}^+,L_{1,q^{2}}^+,L(Y_{1,q^{-1}}).
\end{equation}
Moreover, a tensor product of those modules is simple if and only if it does not contain both a factor $L_{1,1}^+$ and a factor $L(Y_{1,q^{-1}})$ (the others are in so-called pairwise \emph{general position}). Thus, in this situation, the simple modules are in bijection with the cluster monomials:
\begin{equation*}
\begin{array}{ccc}
\left\lbrace \begin{array}{c}
\text{simple modules} \\
\text{ in } \cO_1^+
\end{array}\right\rbrace & \longleftrightarrow & \left\lbrace \begin{array}{c}
\text{commutative quantum cluster}\\
\text{monomials in } \A(\Lambda_1,\Gamma_1)
\end{array}\right\rbrace  \\
\left(L_{1,q^{-2}}^+\right)^{\otimes \alpha}\otimes \left(L_{1,1}^+\right)^{\otimes \beta}\otimes
\left(L_{1,q^{2}}^+\right)^{\otimes \gamma} &  \mapsto & K^{\alpha} F^{ \beta}
K'^{ \gamma},\\
\left(L_{1,q^{-2}}^+\right)^{\otimes \alpha'}\otimes \left(L_{1,q^{2}}^+\right)^{\otimes \beta'}\otimes
L(Y_{1,q^{-1}})^{\otimes \gamma'} &  \mapsto & K^{ \alpha'} K'^{ \beta'} E^{ \gamma'}.
\end{array}
\end{equation*}

\vspace*{1cm}

\bibliographystyle{alpha}
\bibliography{article}

\flushleft{
\textsc{Université Paris-Diderot, \\
CNRS Institut de mathématiques de Jussieu-Paris Rive Gauche, UMR 7586,} \\
Bâtiment Sophie Germain, Boite Courrier 7012, \\
8 Place Aurélie Nemours - 75205 PARIS Cedex 13,\\
E-mail: \texttt{lea.bittmann@imj-prg.fr}
}

\end{document}